\documentclass[a4paper, 12pt]{amsart}

\usepackage[a4paper]{geometry}
\usepackage{amssymb}
\usepackage{amsthm}
\usepackage{amsmath}
\usepackage{graphicx}
\usepackage{hyperref}
\usepackage{multirow}
\usepackage{array}
\usepackage[colorinlistoftodos]{todonotes}
\usepackage[english]{babel}
\usepackage{lmodern}
\usepackage{comment}
\usepackage{bbold}
\usepackage{mathtools}
\usepackage{enumerate}
\usepackage{cite}

\newcommand{\genlegendre}[4]{%
	\genfrac{(}{)}{}{#1}{#3}{#4}%
	\if\relax\detokenize{#2}\relax\else_{\!#2}\fi
}

\allowdisplaybreaks

\newcommand{\bartosz}[1]{\textit{{\color{red}$\diamondsuit\diamondsuit\diamondsuit$ Bartosz: #1}}}

\usepackage[utf8]{inputenc}
\usepackage[T1]{fontenc}

\newtheorem{theorem}{Theorem}[section]
\newtheorem{lemma}[theorem]{Lemma}
\newtheorem{corollary}[theorem]{Corollary}
\newtheorem*{conjecture}{Conjecture}

\theoremstyle{definition}
\newtheorem{definition}[theorem]{Definition}
\newtheorem{proposition}[theorem]{Proposition}

\newtheorem{example}[theorem]{Example}

\theoremstyle{remark}
\newtheorem{remark}[theorem]{Remark}
\newcommand{\longcomment}[1]{}

\DeclareMathOperator{\charac}{char}

\DeclareMathOperator{\Res}{Res}

\DeclareMathOperator{\Gal}{Gal}
\DeclareMathOperator{\GalQ}{Gal(\overline{\mathbb{Q}}/\mathbb{Q})}

\DeclareMathOperator{\OO}{\mathcal{O}}

\newcommand{\whichbold}[1]{\mathbb{#1}} 

\newcommand{\F}{\whichbold{F}}
\renewcommand{\P}{\whichbold{P}}

\newcommand{\ZZ}{\whichbold{Z}}

\newcommand{\RR}{\whichbold{R}}
\renewcommand{\AA}{\whichbold{A}}
\newcommand{\QQ}{\whichbold{Q}}
\newcommand{\Fq}{\whichbold{F}_{q}}
\newcommand{\Fp}{\whichbold{F}_{p}}



\numberwithin{equation}{section}
\author{Matija Kazalicki}
\address{Department of Mathematics\\ 
	University of Zagreb\\
	 Bijeni\v{c}ka cesta 30\\
	  10000 Zagreb\\
	  Croatia}
\email{matija.kazalicki@math.hr}

\author{Bartosz Naskręcki}
\address{Faculty of Mathematics and Computer Science, Adam Mickiewicz
	University Poznań, Poland; ORCID
	0000-0003-2484-143X}
\email{bartosz.naskrecki@amu.edu.pl}

\title[Second moments and the Bias conjecture]{Second moments and the Bias conjecture for the family of cubic pencils}

\makeatletter
\renewcommand{\tocsection}[3]{%
	\indentlabel{\@ifnotempty{#2}{\bfseries\ignorespaces#1 #2\quad}}\bfseries#3}
\renewcommand{\tocsubsection}[3]{%
	\indentlabel{\@ifnotempty{#2}{\ignorespaces#1 #2\quad}}#3}

\providecommand\@dotsep{4.5}
\def\@tocline#1#2#3#4#5#6#7{\relax
	\ifnum #1>\c@tocdepth 
	\else
	\par \addpenalty\@secpenalty\addvspace{#2}%
	\begingroup \hyphenpenalty\@M
	\@ifempty{#4}{%
		\@tempdima\csname r@tocindent\number#1\endcsname\relax
	}{%
		\@tempdima#4\relax
	}%
	\parindent\z@ \leftskip#3\relax \advance\leftskip\@tempdima\relax
	\rightskip\@pnumwidth plus1em \parfillskip-\@pnumwidth
	#5\leavevmode\hskip-\@tempdima{#6}\nobreak
	\leaders\hbox{$\m@th\mkern \@dotsep mu\hbox{.}\mkern \@dotsep mu$}\hfill
	\nobreak
	\hbox to\@pnumwidth{\@tocpagenum{\ifnum#1=1\bfseries\fi#7}}\par
	\nobreak
	\endgroup
	\fi}
\AtBeginDocument{%
	\expandafter\renewcommand\csname r@tocindent0\endcsname{0pt}
}
\def\l@subsection{\@tocline{2}{0pt}{2.5pc}{5pc}{}}
\makeatother

\begin{document}

\begin{abstract}
For a $1$-parametric family $\mathcal{E}_{U}$ of elliptic curves over $\QQ$  and a prime $p$, consider the second moment sum $M_{2,p}(\mathcal{E}_U)=\sum_{u\in\mathbb{F}_{p}}a_{u,p}^2$, where 
$a_{u,p}=p+1-\#\mathcal{E}_{u}(\mathbb{F}_{p})$. 
Inspired by Rosen and Silverman's proof of the Nagao conjecture, which relates the first moment of a rational elliptic surface to the rank of the Mordell-Weil group of the corresponding elliptic curve, S. J. Miller initiated the study of the asymptotic expansion of $M_{2,p}(\mathcal{E}_U)=p^2+O(p^{3/2})$ (which by the work of Deligne and Michel has a cohomological interpretation). 
He conjectured that, similar to the first moment case, the largest lower-order term that does not average to 0 has a negative bias.
In this paper, we provide an explicit formula for the second moment $M_{2,p}(\mathcal{E}_{U})$  of
\begin{equation*}
	\mathcal{E}_{U}:y^2=P(x)U+Q(x),
\end{equation*}
where $\deg P(x),\deg Q(x)\leq 3$. For a generic choice of polynomials $P(x)$ and $Q(x)$ this formula is expressed in terms of the point count of a certain genus two curve. As an application, we prove that the Bias conjecture holds for the pencil of the cubics $\mathcal{E}_U$.
\end{abstract}
\maketitle
\tableofcontents

\section{Introduction}

Let $\mathcal{E}_{U}$ be a $1$-parametric family of elliptic curves over $\QQ$, i.e., a generic fiber of an elliptic surface $\mathcal{E}\rightarrow\mathbb{P}^{1}$ defined over $\QQ$ with a section. For every $u\in\mathbb{P}^1(\mathbb{Q})$ let $\tilde{\mathcal{E}}_{u,p}$ be a minimal model at a rational prime $p$ of $\mathcal{E}_{u}$, and denote by $a_{u,p}$ the number $p+1-\#\tilde{\mathcal{E}}_{u,p}(\mathbb{F}_{p})$. For an element $u\in\mathbb{F}_{p}$ such that $E=\tilde{\mathcal{E}}_{u,p}$ is an elliptic curve over $\mathbb{F}_{p}$ the number $a_{u,p}$ is the usual trace of the Frobenius endomorphism that acts on the $\ell$-adic cohomology group $H^{1}_{et}(E_{\overline{\mathbb{F}}_{p}},\mathbb{Q}_{\ell})$ for $\ell\neq p$.

We consider the $r$-th moment sum associated with the $\mathcal{E}_U$ and prime $p$
\[M_{r,p}(\mathcal{E}_U)=\sum_{u\in\mathbb{F}_{p}}a_{u,p}^r.\]

These sums are well understood, by the work of Deligne \cite{Deligne_Modular}, if the family $\mathcal{E}_U$ is the universal elliptic curve \cite{Scholl_Modular_motives}, for example, associated to the modular group $\Gamma_1(N)$ - then they can be described as traces of Hecke operators acting on the space of weight $k \leq r+2$ cusp forms $S_{k}(\Gamma_1(N))$. 

Even moment sums for a two-parametric family of elliptic curves $y^2=x^3+ax+b$, $a,b\in\mathbb{F}_{p}$ were considered in the work of Birch \cite{Birch}. In this paper, Birch gave an estimate for the highest order term and proved by using the Selberg trace formula that these sums are polynomial expressions in $p$, computed from the traces of Hecke operators acting on the full level modular group. 
These ideas were discussed in the context of vertical and horizontal Sato-Tate distributions by Katz in \cite{Katz_lectures}.

Dujella and Kazalicki in  \cite{Dujella_Kazalicki_ANT} related some higher moment sums to Diophantine $m$-tuples. Another case was considered in \cite{Kazalicki_Naskrecki_JNT}.
If the family $\mathcal{E}_U$ is not associated to the modular group as above or multiparametric like in the case of Birch, then although often there is no simple closed formula for the $r$-th moments $M_{r,p}(\mathcal{E}_U)$, one can study their averages.

It is natural to extend the definition of the second moment to include the fiber at infinity $\mathcal{E}_\infty$. We define $$\tilde{M}_{r,p}(\mathcal{E}_U)=M_{r,p}(\mathcal{E}_U)+a_{\infty,p}^r,$$
where $a_{\infty,p}=p+1-\#\mathcal{E}_{\infty}(\mathbb{F}_{p})$. 
We can further generalize and define $a_{u,q}=q+1-\#\mathcal{E}_{u}(\mathbb{F}_{q})$
where $q$ is any prime power. At primes of bad reduction the symbol $a_{u,q}$ lacks a natural cohomological interpretation but is natural to include it in this form into the computations. In analogy, we define $\tilde{M}_{r,q}$ to be the sum of $a_{u,q}^r$ contributions for all $u\in\mathbb{P}^{1}(\Fq)$. 

The sums $M_{r,q}(\mathcal{E}_{U})$ are \textsl{motivic} in the following sense. One can attach to each sequence $\mathcal{M}_{r}(\mathcal{E}_{U})_{p}=(M_{r,p^i})_{i}=(M_{r,p^i}(\mathcal{E}_{U}))_{i}$ for a given prime $p$ and positive integer $i$ its zeta function
\[Z(\mathcal{M}_{r}(\mathcal{E}_{U})_{p},T)=\exp\left(\sum_{i=1}^{\infty}\frac{M_{r,p^i}}{i} T^i\right)\]
which is rational, i.e., $Z(\mathcal{M}_{r}(\mathcal{E}_{U})_{p},T)\in\mathbb{Q}(T)$. In consequence, there exists a set of algebraic integers $\alpha_{i,j}^{(p)}\in\overline{\mathbb{Q}}$ such that for every prime $p$ and exponent~$m$

\begin{equation}\label{eq:motivic_sum}
	M_{r,p^m} = \sum_{i=0}^{2m}\sum_{j}(\alpha_{i,j}^{(p)})^m
\end{equation}
where 
\begin{equation}\label{eq:alpha_weights}
	|\alpha_{i,j}^{(p)}| = p^{i/2}.
\end{equation}

In a general context, one would also assume something about the meromorphy of the $L$-function obtained from the product of all zeta functions as a function on the complex plane, but we do not need these properties in what follows.

The rationality of $Z(\mathcal{M}_{r}(\mathcal{E}_{U})_{p},T)$ follows from a general result \cite{Michel} about the trace of Frobenius endomorphism acting on the cohomology group
\[
W_{r,\ell} = H^1_{et}(\mathbb{P}^{1}_j\otimes \overline{\QQ}, i_*\textrm{Sym}^{r}\mathcal{E}_\ell)
\]
where $i_*\textrm{Sym}^{r}\mathcal{E}_\ell$ is an $\ell$-adic sheaf associated with the family $\mathcal{E}\rightarrow\mathbb{P}^1$. In general, it is difficult to study the Frobenius trace on such a cohomology group, cf. \cite{Scholl_Modular_motives},\cite{Deligne_Weil_II}. Since $W_{r,\ell}$ in general is not pure \cite{Michel} it follows that it does not correspond to a pure motive. Instead, we expect a ``mixed motivic'' decomposition which we will make precise in the follow-up.

\subsection{A pencil of cubics and its second moment}
The main goal of this paper is to describe explicitly the motivic decomposition of the second moment for a general family of cubic pencils. In a parallel paper \cite{Kazalicki_Naskrecki_JNT} we have studied a detailed decomposition of the second moment sum with the help of point counts on a certain threefold fibred in K3 surfaces. 

Encouraged by special examples from \cite{Mackall_bias}, we discovered that one can obtain a uniform second moment description for every cubic pencil with the help of a special conic fibration on the naturally defined threefold. We use this formula in Theorem \ref{thm:main2} to compute explicitly the bias of the second moment of the family.  

In particular, for the pencil of cubics
\begin{equation}\label{eq:0}
	\mathcal{E}_{U}:y^2=P(x)U+Q(x),
\end{equation}
where $\deg P(x),\deg Q(x)\leq 3$ there is a formula \eqref{eq:main} for the second moment $\tilde{M}_{2,q}(\mathcal{E}_U)$ of the family \eqref{eq:0} for typical choice of polynomials $P(x)$ and $Q(x)$ which is given in terms of number of $\F_q$-rational points on a certain genus two curve $\overline{D}$ associated to $P(x)$ and $Q(x)$. To make this formulation more precise, we need to introduce a point count on a naturally attached threefold.

Let $M_{aff}$ be the threefold associated to the Kummer surface $Kum(\mathcal{E}_U\times \mathcal{E}_U)$ over $\QQ(U)$
$$M_{aff}: \left( P(x_1)U+Q(x_1) \right)\cdot \left(P(x_2)U+Q(x_2)\right)=y^2 \subset \mathbb{A}^1\times \mathbb{A}^1 \times \mathbb{A}^1\times \mathbb{A}^1$$
and let
$$M_{\infty}:  P(x_1) P(x_2)=y^2\subset \mathbb{A}^1\times \mathbb{A}^1\times \mathbb{A}^1$$
be a fiber of the threefold $M_{aff}$ at $u=\infty$.
Denote by $\iota_{aff}:M_{aff}\rightarrow \mathbb{A}^{1}\times\mathbb{A}^{1}\times\mathbb{P}^{1}\times\mathbb{A}^{1}$ a map $\iota_{aff}(x_1,x_2,U,y)\mapsto (x_1,x_2,(U:1),y)$ and let
$\iota_{\infty}:M_{\infty}\rightarrow \mathbb{A}^{1}\times\mathbb{A}^{1}\times\mathbb{P}^{1}\times\mathbb{A}^{1}$ be a map $\iota_{\infty}(x_1,x_2,y)\mapsto (x_1,x_2,(1:0),y)$. 
Define
\begin{equation}\label{eq:M}
	M=\iota_{aff}(M_{aff})\cup \iota_{\infty}(M_{\infty}).
\end{equation}

Our starting point is the observation (see Theorem \ref{thm:threefold}) that the number of $\F_q$-rational points on the threefold $M$
is equal to $q^3+q^2+\tilde{M}_{2,q}(\mathcal{E}_U)$ for any prime power $q$.

Furthermore, we consider the morphism $\pi:M \rightarrow \mathbb{A}^2$, \begin{equation}\label{eq:pi_morphism}
\pi(x_1,x_2,(U:s),y)=(x_1,x_2)
\end{equation}
and denote by $M_{x_1,x_2}=\pi^{-1}(x_1,x_2)$ the fiber over $(x_1,x_2)$. We show that $\tilde{M}_{2,q}(\mathcal{E}_U)$ is determined by the degenerate fibers of the
fibration $\pi$. More precisely, if we denote by $$\Delta(x_1,x_2):=P(x_1)Q(x_2)-P(x_2)Q(x_1)$$ the square root of the discriminant of the conic $M_{x_1,x_2}$, and by $$C=M_\infty \cap \pi^{-1}(\Delta)$$ the 2:1 lift of the curve $$\Delta: \Delta(x_1,x_2)=0$$ to the surface $M_\infty: P(x_1)P(x_2)=y^2$ then in Theorem \ref{thm:Second_moment_curve_formula} we show that for an odd prime power $q$ we have
\begin{equation}\label{eq:sec_moment}
\tilde{M}_{2,q}(\mathcal{E}_U)=q\left( -\# \Delta(\Fq)+\#C(\Fq)+ \left[ \sum_{P(x) \equiv 0} \phi_q(Q(x))\right]^2\right)
\end{equation}
for the unique degree $2$ character $\phi_q:\Fq^{\times}\rightarrow\{-1,1\}$.

The formula \eqref{eq:sec_moment} has a reinterpretation in terms of yet another curve. This curve is going to be used to compute the average of $\tilde{M}_{2,p}(\mathcal{E}_U)$ over primes $p$ in some generality. The computation of averages of moments of elliptic curves has a long and rich history. In particular, the first moment average was intensively studied by Rosen and Silverman \cite{Rosen_Silverman}. For the second moment we discuss in detail the application of averages in the context of the Bias conjecture \ref{conj:bias_conjecture}.
For a generic choice of polynomials $P(x)$ and $Q(x)$ (in the sense of Section \ref{sec:typical_polynomials}) we simplify the formula above by introducing a certain genus two curve which is related to $C$ and $\Delta$. 
 Our~choice of the ''generic'' polynomials in Section \ref{sec:typical_polynomials} was inspired by the following observations:

\begin{itemize}
	\item The curve $\Delta(x_1,x_2)=0$ is the union of the line $\ell: x_1=x_2$, and another curve $\tilde{\Delta}: \tilde{\Delta}(x_1,x_2)=0$, where $\tilde{\Delta}(x_1,x_2)=\Delta(x_1,x_2)/(x_1-x_2)$. The curve $\tilde{\Delta}$ can be both irreducible and reducible but in the geometrically irreducible and reduced case its genus is at most $1$.
	\item  The curve $\tilde{C}=M_\infty \cap \pi^{-1}(\tilde\Delta)$ has geometric genus at most $3$ if it is geometrically irreducible and geometrically reduced.
\end{itemize}
\begin{definition}
	Let $K$ be a field of characteristic not equal to $2$. A pair of curves $(\tilde{\Delta},\tilde{C})$ is {\bf $K$-typical} if and only if both curves are geometrically reduced and irreducible and the geometric genus is $1$ and $3$, respectively. We say that the polynomials $P$ and $Q$ are $K$-typical if the condition holds for the corresponding curves.
\end{definition}
	
In Section \ref{sec:generic}, we reformulate what it means for the pair of polynomials $(\tilde{\Delta},\tilde{C})$ to be $K$-typical.
In Lemmas \ref{lem:Delta_sing}, \ref{lem:C_sing} and Corollary \ref{cor:singular}, we study singular points of $\tilde{\Delta}$ and $\tilde{C}$ as well as the fields of definition of their resolutions in the smooth models $\overline{C}$ and $\overline{\Delta}$ of $\tilde{C}$ and $\tilde{\Delta}$ respectively.
Therefore, from Corollary \ref{cor:final} and Corollary \ref{prop:D} we obtain that for prime powers $q$ coprime with $2$ the formula \eqref{eq:sec_moment} takes the following form
\begin{align}
\tilde{M}_{2,q}(\mathcal{E}_U)&=q\left( \#\overline{C}(\Fq)-\#\overline{\Delta}(\Fq)+q-\#S(\Fq)\right),\nonumber\\
&= q\left(\#\overline{D}(\Fq)+q-\#S(\Fq)\right)\label{eq:main}
\end{align}
where $S \subset \Delta$ is an intersection of the line $\ell$. The curve $\tilde{\Delta}$, and $\overline{D}$ is a genus two quotient of $\overline{C}$ by the involution whose restriction to $\tilde{C}$ is given by $\tau_2:(x_1,x_2,y)\mapsto (x_2,x_1,-y)$.

The family $\pi:M\rightarrow\mathbb{A}^2$ of conics can be analysed from a point of view of Chow groups as in \cite{Beauville_Prym}. For a fixed $u$ the fiber $M_{u}$ is typically a rational elliptic surface. 
For each $u$ the existence of a birational map $\xi_{u}:\mathbb{P}^2\dashrightarrow M_{u}$ and hence of a dominant map $\xi:\mathbb{P}^2\times\mathbb{P}^1\dashrightarrow M$ implies that $M$ is uniruled and unirational. 
A complete non-singular model $\tilde{M}$ of $M$ has a non-trivial motive $h^3(\tilde{M})$ which can be related in the case of conic bundles to the Prym variety of the pair of curves $\tilde{C}\rightarrow\tilde{\Delta}$.
The curve $\tilde{\Delta}$ has a natural interpretation as the discriminant curve of the conic bundle $\pi$ and $\tilde{C}$ corresponds to the Fano variety of lines on $\tilde{C}$, cf. \cite{Nagel_Saito}.
In fact, as proved in \cite{Mumford_Prym}, the Prym variety $Prym(\overline{C}/\overline{\Delta})$ has dimension $2$ and has a polarization of type $(1,2)$ and is linked to a genus $2$ curve constructed from $\overline{C}$. In the context of typical pairs $(\tilde{\Delta},\tilde{C})$ the sought after curve is $\overline{D}$.

\subsection{Application to Bias conjecture}
In this section, let $\mathcal{E}_{U}$ denote an arbitrary elliptic curve over $\mathbb{Q}(U)$.
Rosen and Silverman \cite{Rosen_Silverman} have proved a conjecture of Nagao which implies that the first moment $M_{1,p}(\mathcal{E}_{U})$ of a rational elliptic surface is related to the Mordell-Weil rank of the group $\mathcal{E}_{U}(\mathbb{Q}(U))$. 
In precise terms the limit of 
$$\frac{1}{X}\sum_{p\leq X} M_{1,p}(\mathcal{E}_{U}) \frac{\log p}{p}$$ 
for $X\rightarrow\infty$ exists and is equal to $-\mathrm{rank}\ \mathcal{E}_{U}(\mathbb{Q}(U))$.
\begin{remark}
It follows from the Abel summation formula that the limit above is equal to the limit $$\frac{1}{\pi(X)}\sum_{p\leq X} \frac{M_{1,p}(\mathcal{E}_{U})}{p}$$ as $X\rightarrow\infty$, where $\pi(X)$ is the prime-counting function.
\end{remark}
It follows that the numbers $M_{1,p}(\mathcal{E}_{U})/p$ (i.e., the coefficients $a_{u,p}$) have negative bias (and the larger the rank of the family, the greater the bias). 

Inspired by this result, Miller, in his thesis \cite{Miller_thesis} and subsequently with his co-authors in \cite{Miller_et_al_biases} initiated a study of bias for the second moments $M_{2,p}(\mathcal{E}_{U})$.  For families of hyperelliptic curves the Bias conjecture was studied in \cite{Higher_genus_bias}.

A formula of the form \eqref{eq:motivic_sum} implies a canonical stratification of the second moment $$M_{2,p}(\mathcal{E}_{U})=\sum_{i=0}^{4}f_{i}(p)$$ with the property that $|f_{i}(p)|\leq C p^{i/2}$ where the constant $C$ depends only on the choice of the family $\mathcal{E}_{U}$. Namely, we define $$f_i(p)=\sum_{j}\alpha_{i,j}^{(p)},$$
where $j$ varies over a finite set of size independent of $p$(for almost every $p$).

For the sequence of integers $\{f_{i}(p)\}_{p}$  parametrized by prime numbers $p$, we define the average $\mu(\{f_i(p)\})$ to be the limit
$$\mu(\{f_i(p)\}):=\lim_{X\rightarrow\infty}\frac{1}{\pi(X)}\sum_{p\leq X}\frac{f_{i}(p)}{p^{i/2}},$$ if such a limit exists.
It follows from the result of Michel \cite{Michel} that the leading term $f_4(p)$ of $M_{2,p}$ has a positive average.
\begin{theorem}[Michel]
	For a family of elliptic curves $\mathcal{E}_{U}$ with non-constant $j$-invariant the average $\mu(\{f_{4}(p)\})$ exists and is positive. 
\end{theorem}
Our choice of stratification $M_{2,p}=\sum f_i(p)$ was implicit in \cite{Michel}, where Michel (using the fundamental work of Deligne \cite{Deligne_Weil_II} and Katz \cite{Katz} about $\ell$-adic cohomology and monodromy groups) studied second moments by studying the action of Frobenius on certain cohomology groups.

The terms $f_i(p)$ where $i<4$ are called the \textit{lower order terms} and $i$ is the \textit{degree} of that term. 

In \cite{Miller_et_al_biases} and \cite{Mackall_bias} the authors studied the second moments $M_{2,p}(\mathcal{E}_U)$ of the families of elliptic curves $\mathcal{E}_U$.
Most of their examples are pencils of cubics 
\begin{equation}\label{eq:EU}
	\mathcal{E}_{U}:y^2=P(x)U+Q(x),
\end{equation}
where $\deg P(x),\deg Q(x)\leq 3$ for which they expressed the second moment $M_{2,p}(\mathcal{E}_{U})$ in terms of Legendre symbol sums. For some choices of (in the context of our work degenerate) polynomials $P(x)$ and $Q(x)$ they computed this sum and proposed the following conjecture. 

\begin{conjecture}[Bias conjecture]\label{conj:bias_conjecture}
	Let $\mathcal{E}_U$ be a one-parameter family of elliptic curves over $\QQ(U)$. The averages of the lower order terms in the second moment expansion of $M_{2,p}(\mathcal{E}_{U})$ exist, and the largest lower order term that does not average to $0$ is on the average negative.
\end{conjecture}

\begin{remark}
	The formulation of the Bias conjecture proposed in the previous work (cf. \cite{Miller_et_al_biases},\cite{Higher_genus_bias},\cite{Mackall_bias}) was not precise enough since it did not uniquely define the stratification of $M_{2,p}$ used in a definition of bias. 
\end{remark}
For a generic choice of polynomials $P(x)$ and $Q(x)$, as a consequence of Corollary \ref{cor:final}, we prove the Bias conjecture for the pencil of cubics \eqref{eq:0}.

Consider a smooth projective curve \(C\) of genus \(g>0\) defined over \( \mathbb{Q} \) with a fixed model. For each prime we consider a reduction $C_p$ of the model of $C$ modulo a prime number $p$. For all but finitely many $p$ the scheme $C_p$ is smooth over $\mathbb{F}_{p}$. Let
\[
\theta_p = p + 1 - \#C_p(\mathbb{F}_p).
\]

In our proof, we require the following results concerning the vanishing of the moments of the sequence \( \frac{\theta_p}{\sqrt{p}} \), associated with certain curves of genus one and a specific genus two curve (determined by \( P(x) \) and \( Q(x) \)).

When \( g = 1 \), it follows from the Modularity theorem \cite{Breuil_Conrad_Diamond_Taylor} and a standard analytic argument via the Wiener-Ikehara theorem that the limit
\begin{equation}\label{eq:first_moment}
    \lim_{x \to \infty} \frac{\sum_{p \leq x} \frac{\theta_{p}}{\sqrt{p}}}{\pi(x)}
\end{equation}
is zero. Similarly, the analytic continuation of the symmetric square \( L \)-function of an elliptic curve \cite{Gelbart_Jacquet_SymmetricSquare} implies that the limit
\begin{equation}\label{eq:second_moment}
    \lim_{x \to \infty} \frac{\sum_{p \leq x} \frac{\theta_{p}^2}{p}}{\pi(x)}
\end{equation}
is also zero.

For \( g = 2 \), using the potential automorphy result for genus two curves \cite{Boxer_Calegari_Pilloni}, along with Brauer induction and the theory of cyclic base change \cite{Arthur_Clozel}, it is shown in \cite{Taylor_Noah_distributions} that the first moment \eqref{eq:first_moment} is again zero.

\begin{remark}

These results about the moments follow from the Sato-Tate conjecture. The conjecture holds for elliptic curves over totally real fields \cite{Harris_et_al_Sato_Tate}, and for genus two curves, it has been proved in all but one case cf. \cite{Johansson, Taylor_Noah_distributions}, following the classification of Sato-Tate groups by \cite{Fite_Kedlaya_Rotger_Sutherland_genus2}.

\end{remark}

For the family \eqref{eq:EU}, as a corollary of Theorem \ref{thm:Second_moment_curve_formula}, we prove directly that the sum $M_{2,p}(\mathcal{E}_{U})$ is motivic.

\begin{theorem}\label{thm:motivic_summation}
	Let $P$, $Q$ be polynomials in $\mathbb{Z}[x]$ and let $\mathcal{E}_{U}:P(x)U+Q(x)=y^2$ be a family of curves. For any prime number $p$ the sum $M_{2,p}(\mathcal{E}_{U})$ is motivic.
\end{theorem}

In Section \ref{sec:contribution} we study the average of $\#S(\Fp)=\# \{(x,x)\in\Fp^2: \widetilde{\Delta}(x,x)=0\}$ over the primes. Using some elementary representation theory of finite group, in Proposition \ref{cor:average} we show that this average is equal to the number of irreducible factors into which $S$ decomposes over $\mathbb{Q}$.

Denote by $\overline{D}$ the quotient of $\overline{C}$ by involution whose restriction to $\tilde{C}$ is given by $\tau_2:(x_1,x_2,y)\mapsto (x_2,x_1,-y)$. In the $K$-typical case the curve $\overline{D}$ has geometric genus two by Corollary \ref{prop:D}. 
Taken together, these findings imply our main result.

\begin{theorem}\label{thm:main2}
	Let $P$ and $Q$ be two polynomials of degree at most $3$ over $\mathbb{Q}$ and let $\mathcal{E}_{U}: P(x)U+Q(x)=y^2$ be the associated pencil of cubics. Assume that the associated pair of curves $(\tilde{\Delta},\tilde{C})$ is $\mathbb{Q}$-typical. Let $\overline{D}$ be the genus $2$ curve obtained as the quotient of $\overline{C}$ by the involution $\tau_2$. 
	
 The Bias conjecture holds for the pencil $\mathcal{E}_{U}$ and the bias equals $-m-\delta$ where $m$ is the number of irreducible factors of the polynomial $S$ and $\delta\in\{0,1\}$, vanishing only when the fiber $\mathcal{E}_{\infty}$ is singular. Each number in the set $\{-1,-2,-3,-4,-5\}$ corresponds to an infinite family of non-isomorphic pencils $\mathcal{E}_{U}$.
\end{theorem}

We do not offer a complete description of the bias for the families where $(\tilde{\Delta},\tilde{C})$ are not typical but we have some partial results (see Section \ref{sec:non_typical}) which strongly suggests that the Bias conjecture is true as well in that case. 

 \section*{Acknowledgements}
 The authors were supported by the Croatian Science Foundation under the project no.~IP-2018-01-1313. MK acknowledges the support from the QuantiXLie Center of Excellence, a project co-financed by the Croatian Government and European Union through the European Regional Development Fund - the Competitiveness and Cohesion
 Operational Programme (Grants KK.01.1.1.01.0004 and PK.1.1.02.0004) and support from the Croatian Science Foundation under the project no.~IP-2022-10-5008.
 BN acknowledges the support by Dioscuri program initiated by the Max
 Planck Society, jointly managed with the National Science Centre
 (Poland), and mutually funded by the Polish Ministry of Science and
 Higher Education and the German Federal Ministry of Education and
 Research.
 
 The authors would like to thank the participants of the SFARA seminar, Wojciech Gajda, Sławek Cynk, Paweł Borówka, Jędrzej Garnek and Wojtek Wawrów for constructive discussions about the earlier version of this paper. We are also grateful to the University of Bristol for providing us with the access to Magma cluster CREAM.

We would like to express our gratitude to Francesc Fité for his valuable advice on specific aspects of the automorphy lifting theorems. We also thank the referee for their insightful remarks and suggestions.
 
\section{The second moment and the point count on the Kummer threefold}\label{sec:Miller}

Let $q$ be an odd prime power. In this section, we express the second moment $M_{2,q}(\mathcal{E}_U)$ of family \eqref{eq:0} in terms of $\#M(\F_q)$ the number of $\F_q$-rational points on threefold $M$ defined in \eqref{eq:M}. 

It will be convenient to extend the second moment sum to
include the fiber $\mathcal{E}_\infty$, so we define $$\tilde{M}_{2,q}(\mathcal{E}_U)=M_{2,q}(\mathcal{E}_U)+a_{\infty,q}^2,$$ where $a_{\infty,q}=q-\#\{(x,y)\in \F_q^2: P(x)=y^2\}$. 
The following proposition is required.
\begin{proposition}[\protect{\cite[Thm. 2.1.2]{Berndt_Evans_Williams}}]\label{prop:quadratic_sum}
	Let $q$ be a prime power $p^s$ where $p>2$.
	Let $\phi_{q}:\mathbb{F}_{q}^{\times}\rightarrow\mathbb{C}^{\times}$ be the unique multiplicative character of order $2$. Let~$\alpha,\beta,\gamma\in\mathbb{F}_{q}$ be given elements. Let $\Delta=4\alpha\gamma-\beta^2$, then
	\begin{equation*}
	\sum_{t \in \mathbb{F}_{q}}\phi_{q}(\alpha t^2+\beta t+\gamma)=\left\{\begin{array}{cc}
	-\phi_q(\alpha), & \alpha\neq 0, \Delta\neq 0\\
	(q-1)\phi_{q}(\alpha) ,& \alpha\neq 0, \Delta=0\\
	 -\phi_q(\alpha)& \alpha=0, \Delta\neq 0\\
	  q\phi_{q}(\gamma)& \alpha=0, \Delta=0
	\end{array}\right.
	\end{equation*}
\end{proposition}
\begin{proof}
	Let $S$ denote the sum $\sum_{t \in \mathbb{F}_{q}}\phi_{q}(\alpha t^2+\beta t+\gamma)$. Let us consider the only non-trivial case: $\Delta\neq 0$ and $\alpha\neq 0$. Under these assumptions it follows that $S = \phi_{q}(\alpha)\sum_{t \in \mathbb{F}_{q}}\phi_q(t^2+(4\alpha\gamma-\beta^2)/(4\alpha^2))$.
	Let $A$ denote $(4\alpha\gamma-\beta^2)/(4\alpha^2)$. Notice that $A\neq 0$. We have that $\sum_{t \in \Fq}\phi_q(t^2+A)$ equals $N(x^2=z^2+A)-q$, where $N$ denote the number of $\Fq$-points on the affine curve $x^2=z^2+A$. It follows from \cite[Chap.8, \S 7, Thm.5]{Ireland_Rosen} that $N(x^2=z^2+A)=q+\phi_{q}(-1)J(\phi_q,\phi_q)$ where $J(\phi_q,\phi_q)$ denotes the Jacobi sum. By \cite[Chap.8, \S 3, Thm.1 (c)]{Ireland_Rosen}; \cite[Thm. 2.1.1(c)]{Berndt_Evans_Williams} we have that $J(\phi_q,\phi_q)=-\phi_q(-1)$ from which the proposition follows.
	
\end{proof}

In \cite{Mackall_bias} the authors derive the following formula for the second moment ~$M_{2,p}(\mathcal{E}_{U})$

\begin{equation}\label{eq:1}
	M_{2,p}(\mathcal{E}_U)= p \left[ \sum_{P(x) \equiv 0} \phi_p(Q(x))\right]^2-
	\left[ \sum_{x \in\mathbb{F}_{p}} \phi_p(P(x))\right]^2+p \sum_{\Delta(x,y)\equiv 0}\phi_p(P(x)Q(x)), 
\end{equation}
where $\Delta(x,y)=P(x)Q(y)-P(y)Q(x).$

Essentially, the proof of \eqref{eq:1} in \cite{Mackall_bias} translates to the proof of a total point count on the threefold $M$ defined in \eqref{eq:M}.
\begin{theorem} \label{thm:threefold}
	Let $P, Q\in \mathbb{Z}[x]$ be two polynomials of degree at most $3$. Let $\mathcal{E}_{U}: y^2=P(x)U+Q(x)$. For a prime power $q$, we have
	$$\#M(\Fq) = q^3+q^2 +\tilde{M}_{2,q}(\mathcal{E}_U).$$
\end{theorem}
\begin{proof}
	For $2|q$ observe that $\tilde{M}_{2,q}(\mathcal{E}_{U})=0$ and $\#M(\Fq)=q^3+q^2$ because $t\mapsto t^2$ is an automorphism of $\mathbb{F}_{q}$.
	Assume from now on that $2\nmid q$.
	It is easy to see that $\#M_\infty(\Fq)=a_{\infty,q}^2+q^2$.
To compute $$\#M(\Fq) =  a_{\infty,q}^2+q^2+ \sum_{x_1,x_2\in \Fq}\sum_{u \in \Fq}\left(1+\phi_q\left( (P(x_1)u+Q(x_1))(P(x_2)u+Q(x_2)\right) \right)$$ we apply Proposition \ref{prop:quadratic_sum}
to the polynomial $\left( P(x_1)U+Q(x_1) \right)\cdot \left(P(x_2)U+Q(x_2)\right) \in \QQ(x_1,x_2)[U]$  (note that discriminant condition becomes $\Delta(x_1,x_2) \in\Fq^{\times}$). We obtain
\begin{align*}
\#M(\Fq) &= \sum_{\substack{x_1,x_2 \in \Fq \\ \Delta(x_1,x_2) \not \equiv 0}}\left(q-\phi_q(P(x_1)P(x_2)) \right)+ \sum_{\substack{x_1,x_2 \in \Fq \\ \Delta(x_1,x_2) \equiv 0\\P(x_1) \not \equiv 0}}\left( q+(q-1)\phi_q(P(x_1)P(x_2)) \right)\\
&+\sum_{\substack{x_1,x_2 \in \Fq\\ \Delta(x_1,x_2) \equiv 0\\P(x_1) \equiv 0}}\left(q+ q \phi_q(Q(x_1)Q(x_2))\right)+a_{\infty,q}^2+q^2\\
&= \left(q^2-\# \Delta(\Fq)\right)q - \sum_{\substack{x_1,x_2 \in \Fq \\ \Delta(x_1,x_2) \not \equiv 0}}\phi_q(P(x_1)P(x_2))+\# \Delta(\Fq) q\\&+q\sum_{\substack{x_1,x_2 \in \Fq \\ \Delta(x_1,x_2)\equiv 0}}\phi_q(P(x_1)P(x_2))-\sum_{\substack{x_1,x_2 \in \Fq \\ \Delta(x_1,x_2)\equiv 0}}\phi_q(P(x_1)P(x_2))\\& + q \sum_{\substack{x_1,x_2 \in \Fq \\ \Delta(x_1,x_2)\equiv 0\\P(x_1) \equiv 0}}\phi_q(Q(x_1)Q(x_2))+a_{\infty,q}^2+q^2\\
&= q^3-
\left[ \sum_{x \in \Fq} \phi_q(P(x))\right]^2+q\sum_{\substack{x_1,x_2 \in \Fq \\ \Delta(x_1,x_2)\equiv 0}}\phi_q(P(x_1)P(x_2))\\&+q \sum_{\substack{x_1,x_2 \in \Fq \\ \Delta(x_1,x_2)\equiv 0\\P(x_1) \equiv 0, P(x_2) \equiv 0 }}\phi_q(Q(x_1)Q(x_2))
+a_{\infty,q}^2+q^2
\end{align*}
hence the claim follows.	
\end{proof}

The proof of the previous theorem is essentially an example of the double counting method. The points of the set $M(\F_q)$ were counted in two ways - by using the fibration of $M$ by elliptic K3 surfaces defined by function $U$ one can express $\# M(\F_q)$ in terms of second moment sum $\sum_{u \in \F_q} a_{u,q}^2$, while using fibration by rational elliptic surfaces defined by $x_1$ one gets Miller's formula \eqref{eq:1}. Our main idea is to use for counting the third fibration $\pi:M \rightarrow \AA^2$ defined in \eqref{eq:pi_morphism}.

\begin{theorem}\label{thm:Second_moment_curve_formula}
	Let $q$ be a prime power such that $2\nmid q$. We have
	$$\#M(\Fq)=q^3+q^2-q \# \Delta(\Fq)+q \#C(\Fq)+ q \left[ \sum_{P(x) \equiv 0} \phi_q(Q(x))\right]^2.$$
	In particular
	$$\tilde{M}_{2,q}(\mathcal{E}_{U})=q\left( -\# \Delta(\Fq)+\#C(\Fq)+ \left[ \sum_{P(x) \equiv 0} \phi_q(Q(x))\right]^2\right).$$
\end{theorem}
\begin{proof}
In the proof of Theorem \ref{thm:threefold} we expressed $\#M(\Fq)$ as the sum $$\sum_{(x_1,x_2)\in \mathbb{A}^2(\Fq)} \#\pi^{-1}(x_1,x_2),$$
where $\pi$ is the morphism \eqref{eq:pi_morphism}. 
Let  $M_{x_1,x_2}$ denote the fiber $\pi^{-1}(x_1,x_2)$ over  $(x_1,x_2)\in \mathbb{A}^2(\Fq)$. Note that $\Delta(x_1,x_2)$ is the square-root of the discriminant of $\left( P(x_1)U+Q(x_1) \right)\cdot \left(P(x_2)U+Q(x_2)\right)$ with respect to $U$.

\begin{enumerate}
	\item [a)] Assume $\Delta(x_1,x_2)\ne 0$. 	
	In this case $M_{x_1,x_2}$ is a geometrically irreducible conic. It follows from formula \eqref{prop:quadratic_sum} that an affine part of $M_{x_1,x_2}$ has $q-\phi_q(P(x_1)P(x_2))$ $\Fq$-rational points, hence $\#M_{x_1,x_2}(\Fq)=q+1$. 
	\item[b)] Assume $\Delta(x_1,x_2)=0$ and $P(x_1)\ne 0$. From a definition of $\Delta(x_1,x_2)$ it follows that for every $u\in \Fq$
	$$\left( P(x_1)u+Q(x_1) \right)\cdot \left(P(x_2)u+Q(x_2)\right)=\frac{P(x_2)}{P(x_1)} \left(P(x_1)u+Q(x_1)\right)^2.$$
	Therefore, $\#M_{x_1,x_2}(\Fq)=1$ if $ \phi_q(P(x_1)P(x_2))=-1$, $\#M_{x_1,x_2}(\Fq)=2(q-1)+1+2=2q+1$ if $ \phi_q(P(x_1)P(x_2))=1$ and  $\#M_{x_1,x_2}(\Fq)=q+1$ if $P(x_2)=0$.
	
	This result can be interpreted in the following way. Denote by $C$ a curve $\pi^{-1}(\Delta) \cap M_\infty$. In particular $C(\Fq)=\{ (x_1,x_2,y)\in \Fq^3: P(x_1)P(x_2)=y^2 \textrm{ and } \Delta(x_1,x_2)=0\}$. If we denote by $n(x_1,x_2)$ the number of $\Fq$-rational points on $C$ above $(x_1,x_2)$, then $\#M_{x_1,x_2}(\Fq) = n(x_1,x_2) q + 1$.
	
	\item[c)] Assume $\Delta(x_1,x_2)=0$, $P(x_1)=0$ and $P(x_2)\ne 0$.
	Similar to case $b)$, we have that $\#M_{x_1,x_2}(\Fq) = n(x_1,x_2) q + 1$.
	\item[d)] Assume $\Delta(x_1,x_2)=0$, $P(x_1)=0$ and $P(x_2)=0$. Then $\#M_{x_1,x_2}(\Fq)$ is equal to
	$$\#\{(x_1,x_2,u,y)\in \Fq^4: Q(x_1)Q(x_2)=y^2\}+1=\left( 1+\phi_q(Q(x_1)Q(x_2))\right)q+1.$$
\end{enumerate}	
	The total contribution of elements $(x_1,x_2)$ from the case a) is
	$$\sum_{\Delta(x_1,x_2)\ne 0} \# M_{x_1,x_2}(\Fq)=(q^2-\#\Delta(\Fq))(q+1),$$
	while the total contribution of elements $(x_1,x_2)$ in b), c) and d) is equal to
	$$\# C(\Fq)q + \#\Delta(\Fq) + \sum_{\substack{(x_1,x_2)\in \Delta(x_1,x_2)(\Fq)\\P(x_1)=P(x_2)=0}}\left(-(q+1)+\left( 1+\phi_q(Q(x_1)Q(x_2))\right)q+1 \right).$$
	By combining all contributions, we arrive at the desired result.
\end{proof}

As a consequence of this result, we can prove Theorem \ref{thm:motivic_summation}.
\begin{proof}[Proof of Theorem \ref{thm:motivic_summation}]
	For prime $p=2$ the statement is trivial. For $p>2$ the claim follows from Theorem \ref{thm:Second_moment_curve_formula} and Dwork's theorem on the rationality of the zeta function of a variety defined over $\mathbb{F}_{p}$. The zeta function $Z(X,T)$ associated with the variety $X$ over $\mathbb{F}_{p}$ belongs to $\mathbb{Q}(T)$. Let $\phi_p$ be the unique order $2$ character on $\mathbb{F}_{p}^{\times}$. To~conclude, we need to argue that $A=\left[ \sum_{P(x) \equiv 0} \phi_p(Q(x))\right]^2$,  originates from the point count on a variety. One can write down a variety $V_{A}:\{(x,y):P(x)=0,Q(x)=y^2\}$ for which $(\#V_{A}(\mathbb{F}_p)-\#\{P=0\}(\mathbb{F}_p))^2= A$ holds. The same argument works for any prime power $q=p^s$. Hence, the zeta function of the second-moment sum $M_{2,p}(\mathcal{E}_{U})$ is the product 
	$$\frac{ Z(C,pT)\cdot Z(V_A\times V_A,pT)\cdot Z(P\times P,pT)}{Z(\Delta,pT)\cdot Z(V_A\times P,pT)^2}.$$
\end{proof}
	
Let $P,Q\in\Fq[x]$ be two non-zero polynomials.
Curve $\Delta=0$ is the union of the line $\ell: x_1=x_2$ and the curve $\tilde\Delta=0$, where  $\tilde\Delta(x_1,x_2)=\Delta(x_1,x_2)/(x_1-x_2)$. Let $S$ denote their intersection, defined as 
$$S(\Fq)=\{(x,x)\in \Fq^2:\tilde{\Delta}(x,x)=0\},$$
By abuse of notation let $P=\{x: P(x)=0\}$ denote the set of zeros of the polynomial $P(x)$, and by $$P \cap S = \{x:P(x)=\tilde{\Delta}(x,x)=0\}$$
the intersection of $P$ with $S$. 
Let $\tilde{C}=M_{\infty} \cap \pi^{-1}(\tilde{\Delta})\subset C$.

\begin{proposition}\label{prop:Curve_defect}
	Let $q$ be an odd prime power and let $P,Q\in \Fq[x]$ be non-zero polynomials. We have
\begin{align*}	
 \#C(\Fq)-\# \Delta(\Fq)&=\#\tilde{C}(\Fq)-\#\tilde{\Delta}(\Fq)+q-\#S(\Fq)\\&-\#P(\Fq)+\#(P\cap S)(\Fq).
\end{align*} 
\end{proposition} 
\begin{proof}
We have that $$\#C(\Fq)=\#\tilde{C}(\Fq)+\# (\pi^{-1}(\ell)\cap M_\infty)(\Fq)-\# (\pi^{-1}(S))(\Fq).$$ Note that $\#\Delta(\Fq)=q+\#\tilde{\Delta}(\Fq)-\#S(\Fq)$ and  $$\# (\pi^{-1}(\ell)\cap M_\infty)(\Fq)=2q-\#\{P=0\},$$ hence
\begin{align*}
	\#C(\Fq)-\# \Delta(\Fq) &=\#\tilde{C}(\Fq) - \#\tilde{\Delta}(\Fq)+(2q-q)-\#P(\Fq)\\
	&-\#S(\Fq)+\#(P\cap S)(\Fq),
\end{align*}	
since $\# (\pi^{-1}(S))(\Fq)=2\#S(\Fq)-\#(P\cap S)(\Fq)$. The claim follows.
\end{proof}

\section{Typical case}\label{sec:typical_polynomials}
In the introduction, we formulated sufficient conditions for the Bias conjecture to hold. The pair of curves $(\tilde{\Delta},\tilde{C})$ is $K$-typical if the respective curves have geometric genus equal to $3$ and $1$.

In this section, we introduce a more technical notion of genericity of the pair of curves $(\tilde{\Delta},\tilde{C})$. We prove towards the end of this section that both notions are actually the same. The latter is used for the computations in further sections.

\subsection{Genericity} \label{sec:generic}
We work over a base field $K$ of characteristic not equal to $2$. Let~$a_0,a_1,a_2,a_3$, $b_0,b_1,b_2,b_3$ and $x_1,x_2$ be variables in the polynomial ring \[R=K[a_0,a_1,a_2,a_3,b_0,b_1,b_2,b_3,x_1,x_2].\] 
Let $P(x)$ denote the polynomial $a_3x^3+a_2 x^2+a_1x+a_0$ and let $Q(x)$ denote the polynomial $b_3 x^3+b_2 x^2+b_1 x+b_0$.
Let $\widetilde{\Delta}(x_1,x_2)$ denote the polynomial $\frac{P(x_1)Q(x_2)-P(x_2)Q(x_1)}{x_1-x_2}$, $\delta_{1} = \partial_{x_{1}}\widetilde{\Delta}(x_1,x_2)$ be the partial derivative of $\widetilde{\Delta}(x_1,x_2)$ with respect to $x_1$ and let $\delta_{2} = \partial_{x_{2}}\widetilde{\Delta}(x_1,x_2)$ be the partial derivative with respect to $x_2$. 
Let $I$ denote the ideal spanned by $\widetilde{\Delta},\delta_{1}$ and $\delta_{2}$ in the polynomial ring $R$. 

\begin{proposition}
	The closed subscheme $\mathcal{C}_{fin}=V(I)$ of the affine space $\textrm{Spec}(R)$ is geometrically reduced and decomposes into the union of the following $7$ dimensional varieties:
	\begin{itemize}
		\item $\mathcal{D}=\textrm{Spec}(R/(\mu_{i,j}:0\leq i<j\leq 3))$ where $\mu_{i,j}$ is the $2\times 2$ minor $a_i b_j-a_jb_i$ of the matrix $\left(\begin{array}{cccc} a_0 & a_1 & a_2 & a_3 \\ b_0 & b_1 & b_2 & b_3\end{array}\right),$
		\item $\mathcal{V}=\textrm{Spec}(R/(P(x_2),Q(x_2),x_1-x_2)),$
		\item $\mathcal{N}_{1}=\textrm{Spec}(R/(P(x_2),Q(x_2),x_1(\mu_{1,3} + 
		2 \mu_{2,3} x_2)+(\mu_{1,2}+ 
		2 \mu_{1,3} x_2 + \mu_{2,3} x_2^2))),$
		\item $\mathcal{N}_{2}=\textrm{Spec}(R/(P(x_1),Q(x_1),x_2(\mu_{1,3} + 2 \mu_{2,3}x_1)+(\mu_{1,2} + 2 \mu_{1,3} x_1 + \mu_{2,3} x_1^2))),$
		\item $\mathcal{S}=\textrm{Spec}(R/(\mu_{0,1}- 3\mu_{0,3}x_2^2 - \mu_{1,3} x_2^3,
		\mu_{0,2}+3\mu_{0,3}x_2 -\mu_{2,3} x_2^3,
		\mu_{1,2}+3\mu_{1,3}x_2 +3\mu_{2,3}x_2^2, x_1-x_2)),$
	\end{itemize}
	\[\mathcal{C}_{fin} = \mathcal{D}\cup \mathcal{V}\cup \mathcal{N}_{1}\cup\mathcal{N}_{2}\cup\mathcal{S}.\]
\end{proposition}
\begin{proof}
	Using Gr\"{o}bner bases we obtain a decomposition of the scheme $\mathcal{C}_{fin}$ into the union of schemes defined above. The details of that step and the claim that $\mathcal{C}_{fin}$ is geometrically reduced were checked by an algorithm performed in MAGMA \cite{Magma_computations}.
\end{proof}
\begin{remark}
	Note that $\mathcal{C}_{fin}$ parametrizes polynomials $P$ and $Q$ for which $\widetilde{\Delta}$ is singular over $\overline{K}$ (together with a singular point).
	It follows from the proposition that $\widetilde{\Delta}$ is singular if and only if $P$ and $Q$ have a common root or if the resultant (with respect to $x_2$) of the polynomials defining $\mathcal{S}$ is equal to zero.
\end{remark}

The curve \(\tilde{\Delta}(x_1,x_2)=0\) has degree $4$ when $\mu_{2,3}\neq 0$. In the projective closure it has two singular points  $$\infty_{1}=[x_1:x_2:z]=[1:0:0] \textrm{ and }\infty_{2}=[x_1:x_2:z]=[0:1:0]$$ in $\mathbb{P}^{2}$. When $\mu_{2,3}=0$, the degree of \(\tilde{\Delta}(x_1,x_2)=0\) is smaller or equal to $3$ and the points $\infty_1,\infty_2$ are not necessarily singular.

For a polynomial $f$ in several variables, the initial form of $f$ is the homogeneous polynomial $f_s$ of least degree and such that $f=f_s+f_{s+1}+\ldots$ where each $f_i$ is homogeneous of degree $i$.
\begin{definition}
	Let $X$ be an affine variety defined in the polynomial ring $R$ by the ideal $I$ and let  $x$ be a closed point on $X$. Suppose that the coordinates of the point $x$ all vanish. We define the tangent cone of $X$ at $x$ to be a variety defined by the ideal of initial forms of generators of~$I$.
\end{definition}

The points \(\infty_{i}\) have their tangent cones (under \(\mu_{2,3}\neq 0\)) isomorphic to a union of two lines 
\begin{equation}\label{eq:tangent_cone_inf}
X^2-(\mu_{1,3}^2-4\mu_{0,3}\mu_{2,3})Y^2=0.
\end{equation}
The slopes of lines are generically distinct except when the characteristic of $K$ equals $2$. We define the following as two further closed subschemes of $\textrm{Spec}(R)$:  $$\mathcal{P}_1=\textrm{Spec}(R)/(\mu_{1,3}^2-4\mu_{0,3}\mu_{2,3}) \textrm{ and } \mathcal{P}_{2}=\textrm{Spec}(R)/(\mu_{2,3})$$ which contains the scheme $\mathcal{D}$.  Let $\mathcal{C}_{inf}$ denote the union $\mathcal{P}_{1}\cup\mathcal{P}_{2}$.

Let $\varrho:\textrm{Spec}(R)\rightarrow \textrm{Spec}(K[a_0,\ldots,b_3])$ be the natural projection map. 
\begin{definition}\label{def:K_generic_tuple}
	A pair $(\tilde{\Delta},\tilde{C})$ is $K$-generic if the corresponding tuple $$(a_0,a_1,a_2,a_3,b_0,b_1,b_2,b_3)$$ does not belong to the set $$\varrho(\mathcal{C}_{fin}\cup \mathcal{C}_{inf})(K).$$ 
\end{definition}

\begin{remark}\label{rem:genericity}
	In other words, a pair $(\tilde{\Delta},\tilde{C})$ is $K$-generic if and only if $\tilde{\Delta}$ is a nonsingular curve of degree $4$ for which the tangent cones of two additional singular points in the projective closure are unions of two distinct lines. 
\end{remark}

A smooth model $\overline{C}$ of the curve $\tilde{C}$ has two (typically) singular models with which we work:
\begin{equation} \label{eq:CC1}
	\tilde{C}:\widetilde{\Delta}(x_1,x_2)=0,\quad P(x_1)P(x_2)=y^2\quad \subset \mathbb{A}^{3},
\end{equation}
\begin{equation}\label{eq:CC2}
	\tilde{C}':\widetilde{\Delta}(x_1,x_2)=0,\quad Q(x_1)Q(x_2)=y^2\quad \subset \mathbb{A}^{3}.
\end{equation} 
The first model is convenient to work with when $a_3\neq 0$ and the latter when $b_3\neq 0$. 
The naive projective closure of each model may contain additional lines contained in the hyperplane at infinity. 
From this point on, for any geometrically reduced and irreducible curve $X$, we denote by $X_{cl}$ the unique irreducible component of its projective closure $X^{cl}$ that equals $X$ in the natural affine chart.

\subsection{Generic is equivalent to typical}
In what follows we will focus on the case when $a_3\neq 0$ since the other case $b_3\neq 0$ provides completely symmetric results (with $P$ being replaced by $Q$ accordingly).

We start by proving properties of the curve $\tilde{\Delta}$ announced in Remark \ref{rem:genericity}.  
\begin{proposition}\label{prop:Delta_tilde_irr}
Let $(\widetilde{\Delta},\widetilde{C})$ be $K$-generic. A curve $\widetilde{\Delta}$ is geometrically irreducible and geometrically reduced. 
\end{proposition}
\begin{proof}
By the genericity assumption, the curve $\widetilde{\Delta}$ is geometrically reduced. The~equation of $\widetilde{\Delta}$ is 
$$\sum_{i=1}^{3}\mu_{i,0}x_2^{i-1} + (\mu_{2,0}+(\mu_{3,0}+\mu_{2,1})x_2+\mu_{3,1}x_2^2)x_1 + (\sum_{i=0}^{2}\mu_{3,i} x_2^i) x_1^2=0.$$
In particular, since $\mu_{2,3}\neq 0$, the curve $\widetilde{\Delta}$ has degree $4$ and its projective closure $\widetilde{\Delta}_{cl}$ has exactly two closed points at infinity $[x_1:x_2:z] = [1:0:0]$ and $[0:1:0]$ which are nodes, cf. \eqref{eq:tangent_cone_inf}. 

If $\widetilde{\Delta}$ were not irreducible, then the polynomial $\widetilde{\Delta}(x_1,x_2)$ would factor over $\overline{K}[x_1,x_2]$. Each factor represents a curve and is of degree at most $3$, hence by the Pl\"{u}cker formula its genus is at most $1$. 
According to the Pl\"{u}cker formula, nodes contribute $-1$ to the geometric genus, so none of the factors can contain both branches along the singular points at infinity.
Suppose $\widetilde{\Delta}(x_1,x_2)$ factors over $\overline{K}[x_1,x_2]$ into an irreducible degree $3$ factor $\delta_{1}$ and a linear factor $\delta_{2}$. 
Suppose that the cubic $\delta_1=0$ is a singular curve. Then it must contain both branches around one of the singularities and the curve $\delta_2=0$ intersects $\delta_1=0$ at the second point at infinity transversally. It follows from the Bezout theorem that the intersection of $\delta_1=0$ with $\delta_2=0$ has degree $3$ and there is yet another transversal point of intersection at the affine patch $z=1$. Hence, this point would be singular on $\widetilde{\Delta}(x_1,x_2)=0$ which is not possible under the genericity assumptions.

Suppose that the cubic $\delta_1=0$ is smooth. Then both points at infinity belong to the intersection of $\delta_1=0$ with $\delta_2=0$ and that intersection is transversal. It follows again from the Bezout theorem that there is another singular point on $\widetilde{\Delta}(x_1,x_2)=0$ which is a contradiction with the genericity assumptions.

If  $\widetilde{\Delta}(x_1,x_2)$ factors over $\overline{K}[x_1,x_2]$ into two irreducible quadratic factors $q_1$ and $q_2$, then both curves $q_1=0$ and $q_2=0$ are smooth, so they meet transversally at both points at infinity. Bezout theorem implies that there should exist at least one more point of intersection, a contradiction with the genericity assumptions.

A similar argument applies to the cases where $\widetilde{\Delta}(x_1,x_2)$ factors over $\overline{K}[x_1,x_2]$ into irreducible factors of degrees  $(2,1,1)$ and $(1,1,1,1)$.
\end{proof}

\begin{lemma}\label{lem:Delta_sing}
Let $(\widetilde{\Delta},\widetilde{C})$ be $K$-generic.
	The subscheme $\tilde{\Delta}_{cl}$ is a singular projective curve in $\mathbb{P}^{2}$. Its set of $\overline{K}$-singular points consists of $[0:1:0]$ and $[1:0:0]$. For both points the tangent cone over $K$ is $K$-isomorphic to $x^2-d y^2=0$ in $\mathbb{A}^{2}_{x,y}$ where 
	\[d=\mu_{1,3}^2-4\mu_{0,3}\mu_{2,3}.\]
	The genus of the smooth projective curve $\overline{\Delta}$ equals 1.
\end{lemma}
\begin{proof}
	Since $\tilde{\Delta}$ is $K$-generic, it follows that $\tilde{\Delta}$ has no $\overline{K}$-singular points. We check from the definition of the tangent cone that for the points $[0:1:0]$ and $[1:0:0]$ it is a union of two lines $x^2-dy^2$ in the affine plane $\mathbb{A}_{x,y}^{2}$.
	Since by our assumption on $P$ and $Q$ we have $d\neq 0$, it follows that the singular points are ordinary double points (nodal singularity or type $A_{1}$ singularity). The degree of the curve $\tilde{\Delta}_{cl}$ equals $4$. By the Pl\"{u}cker formula for an irreducible curve in $\mathbb{P}^{2}$, the geometric genus of $\tilde{\Delta}_{cl}$ equals $(4-1)(4-2)/2-2 = 1$.
\end{proof}

Next, we study the properties of $\tilde{C}$.
\begin{proposition}\label{prop:D_C_generic_implies_C_irr}
Let $(\widetilde{\Delta},\widetilde{C})$ be $K$-generic. The curve $\widetilde{C}$ is geometrically irreducible and geometrically reduced.
\end{proposition}	
\begin{proof}
	The curve $\widetilde{C}$ is geometrically reduced by the genericity assumption. 
Assume that $\widetilde{C}$ is not geometrically irreducible. Hence $\widetilde{C} = \sum_{i} C_i$ is a union of irreducible schemes $C_i$ defined over $\overline{K}$ and such that $C_i\nsubseteq C_j$ for any $i\neq j$. 

We consider two projections: $\pi_1,\pi_2:\mathbb{A}^3\rightarrow \mathbb{A}^2$, the first $\pi_1(x_1,x_2,y)=(x_1,x_2)$ and the second $\pi_2(x_1,x_2,y) = (x_1,y)$. The image of $\widetilde{C}$ under the first projection equals $\widetilde{\Delta}$ and the image under the second projection is a curve $B$ which has the following equation

\[B: AP(x_1)^4+\tilde{S}(x_1)P(x_1)^2 y^2+T(x_1)^3 y^4=0\]
where $A,\tilde{S}(x_1) $ and $T(x_1)$ are explicitly determined:

\begin{equation}\label{eq:A_formula}
A=Res_{x}(P(x),Q(x))=-\mu _{0,3}^3+\mu _{1,2} \mu _{0,3}^2-\mu _{0,1} \mu _{1,3}^2+\left(\mu _{0,1} \left(3 \mu _{0,3}+\mu _{1,2}\right)-\mu _{0,2}^2\right) \mu _{2,3},\end{equation}

\[\tilde{S}(x) =s_0+s_1 x+s_2 x^2+s_3 x^3,\]
\begin{align*}
s_0&=\left(2 \mu _{0,3}-\mu _{1,2}\right) \mu _{0,3}^2+\left(\mu _{0,2}^2-3 \mu _{0,1} \mu _{0,3}\right) \mu _{2,3},\\
s_1&=3 \mu _{1,3} \mu _{0,3}^2-\mu _{0,1} \mu _{1,3} \mu _{2,3}-2 \mu _{0,2} \left(\mu _{1,3}^2-\mu _{1,2} \mu _{2,3}\right),\\
s_2&=-\mu _{1,2} \mu _{1,3}^2-4 \mu _{0,1} \mu _{2,3}^2+\left(3 \mu _{0,3}^2+\mu _{1,2}^2+\mu _{0,2} \mu _{1,3}\right) \mu _{2,3},\\
s_3&=-\mu _{1,3}^3+\mu _{2,3} \left(\mu _{1,2}-3 \mu _{3,0}\right) \mu _{1,3}-2 \mu _{0,2} \mu _{2,3}^2,
\end{align*}
 
 \[T(x)=-\mu_{0,3}-\mu_{1,3}x-\mu_{2,3}x^2,\]
 
 \begin{equation}\label{eq:resultant_ST_PQ}
 Res_{x}(\tilde{S}(x),T(x)) = -\mu_{2,3}^3\cdot Res_{x}(P(x),Q(x))^2.
 \end{equation}

In fact, the map $\pi_2\mid_{\tilde{C}}$ has degree $1$ and $B$ is birational to $\tilde{C}$. The inverse map $\pi_2^{-1}\mid _B$ is given~by
\[\pi_2^{-1}\mid _B(x_1,y)=(x_1,g(x_1,y),y)\]
where
\[g(x_1,y)=\frac{1}{ \gamma x_1+\delta}\left(\left(\frac{T(x_1)y}{P(x_1)} \right)^2+\alpha x_1+\beta\right)\]
where $$\alpha=-\mu_{0, 3}\mu_{1, 3} + \mu_{0, 2}\mu_{2, 3}, \quad \beta=-\mu_{0, 3}^2 + \mu_{0, 1}\mu_{2, 3},$$ $$\quad \gamma=\mu_{1, 3}^2 - \mu_{2, 3}(\mu_{0, 3} + \mu_{1, 2}) \textrm{ and } \delta=\mu_{0, 3}\mu_{1, 3} - \mu_{0, 2}\mu_{2, 3}.$$
This map is well-defined. In fact, if $\delta=0=\gamma$, it follows that $A=-2\mu_{0,3}(\mu_{0,3}^2-\mu_{0,1}\mu_{2,3})$, since $\mu_{2,3}\neq 0$. It follows from $\gamma=0$ that $b_0 = \frac{1}{a_3 \mu_{2,3}} (*)$ where $(*)$ is a certain polynomial expression in the variables $a_i$ and $b_j$. Hence 
$$A\mid_{b_0 =  \frac{1}{a_3 \mu_{2,3}} (*)} = \frac{\mu_{0,3}\mu_{1,3}}{\mu_{2,3}}\delta\mid_{b_0 =  \frac{1}{a_3 \mu_{2,3}} (*)}=0.$$
Since $A=Res_{x}(P(x),Q(x))$, the vanishing of the resultant of $P(x)$ and $Q(x)$ would imply that these polynomials have a common root, a contradiction with the $K$-genericity of $(\widetilde{\Delta},\widetilde{C})$.
 
Since $\widetilde{\Delta}$ is geometrically irreducible, then for every $i$ we have $\pi_1(C_i) = \widetilde{\Delta}$. 
Hence, the components $C_i$ differ only on the $y$-coordinate. Assuming that the curve $B$ is geometrically irreducible, we would show that $\pi_2(C_i) = B$ and so prove that all $C_i$ are equal.

To show that $B$ is geometrically irreducible we study the factorization of the polynomial 
 \begin{equation}\label{eq:B_polynomial}
  B(x,y)=A+\tilde{S}(x) y^2+T(x)^3 y^4   
 \end{equation}
which corresponds to a curve birational to $B$ via the change of variables $y\mapsto y\cdot P(x)$. 
 
Suppose that $B(x,y)$ has a factor $q y+p$ where $p,q\in\overline{K}[x]$, $q\neq 0$ and $p,q$ are coprime. If~$p$~would vanish, then $A=0$, but the polynomials $P,Q$ have no common root under the genericity assumptions.
Suppose that $p\neq 0$. It follows that $H=A q^4+\tilde{S}p^2 q^2+T^3 p^4=0$, hence $\deg p =0$ and $q^2\mid T^3$. A polynomial $T$ is of degree $2$ and separable under our assumptions, hence $q=c(x-r_1)^e(x-r_2)^f$ where $r_1,r_2$ are roots of $T$ and $e,f\leq 1$.
Since $\deg A q^4 = 4\deg q$, $\deg \tilde{S}p^2 q^2 \leq 3+2\deg q$ and $\deg T^3 p^4 = 6$, then $H$ cannot vanish when $0\leq \deg q\leq 2$, a contradiction.

Assume that $B(x,y)$ factors into two quadratic expressions in $y$, $B(x,y)=(\alpha+\beta y+\gamma y^2)(\delta+\epsilon y+\phi y^2)$, where $\alpha,\beta,\gamma,\delta,\epsilon,\phi\in \overline{K}[x]$. Since $A\neq 0$ we can assume without loss of generality that $\alpha = A$ and $\delta =1$. It follows that $\beta+A\epsilon = 0$ and we have two separate cases:

\textbullet $\epsilon \neq 0$: In that case $\gamma=A \phi$, so $A\phi^2 = T^3$ and $-A(\epsilon^2-2\phi)=\tilde{S}$. Since~$A\neq 0$ and $T$ is of degree $2$, it follows that $T$ must be a square, hence $\textrm{disc}(T) = \mu_{1,3}^2-4\mu_{0,3}\mu_{2,3}=0$, a contradiction with the genericity assumption.

\textbullet $\epsilon = 0 $: It follows that $\gamma+A\phi = \tilde{S}$ and $\gamma \phi = T^3$. Hence, $\deg \gamma+\deg \phi = 6$ and from the former equation we deduce that in fact $\deg\gamma=\deg\phi = 3$. Notice that $\tilde{S}$ and $T$ cannot have a common root by \eqref{eq:resultant_ST_PQ}. Hence, we found that $B$ is a union of two curves of geometric genus $0$ of the form $A+c(x-r_1)^e(x-r_2)^f y^2=0$ where $c\in\overline{K}$ and $(e,f)\in\{(0,3),(3,0)\}$.  That would imply, that under the mapping $\pi_1:\widetilde{C}\rightarrow\widetilde{\Delta}$ these components would map onto $\widetilde{\Delta}$ which is not possible due to Riemann-Hurwitz formula, Proposition \ref{prop:Delta_tilde_irr} and Lemma \ref{lem:Delta_sing}.
\end{proof}

The following lemma describes the singular points on $\tilde{C}_{cl}$ as well as the fields of the definition of their resolutions.

\begin{lemma}\label{lem:C_sing}
	Let $(\tilde{\Delta},\tilde{C})$ be $K$-generic and let $\psi:\overline{C}\rightarrow \tilde{C}_{cl}$ be the normalization of $\tilde{C}_{cl}$.
	\begin{enumerate}
	\item [a)]If $a_3\neq 0$ and $P$ is separable, then the curve $\tilde{C}$ has six $\overline{K}$-rational nodes at the points $(x_1,x_2,0)$ such that $P(x_1)=P(x_2)=0$ and $x_1\neq x_2$. 
	If $a_3=0$, $b_3\neq 0$ and $Q$ is separable, then the curve $\tilde{C}'$ has six $\overline{K}$-rational nodes at the points $(x_1,x_2,0)$ with $Q(x_1)=Q(x_2)=0$ and $x_1\neq x_2$. The field of definition of every preimage through $\psi$ of a singular point $(x_1,x_2)$ with $P(x_1)=0=P(x_2)$ is $K(x_1,x_2,\sqrt{Q(x_1)Q(x_2)})$.
	
	\item[b)]If $a_3\neq 0$ and $P$ is not separable, then $P(x)=c(x-a)^2(x-b)$, $a\neq b$, $a,b,c\in K$ and the curve $\tilde{C}$ has three $\overline{K}$-rational non-nodal singular points at $(a,b)$, $(b,a)$ and at $(a,a)$. Similarly, we have $3$ singular points when $b_3\neq 0$ and $Q$ has a double root. The non-singular points above $(a,b)$ and $(b,a)$ are defined over $K(\sqrt{Q(a)Q(b)})$ and the non-singular points above $(a,b)$ are defined over $K$.
	
	\item[c)]When either $a_3\neq 0$ or $b_3\neq 0$, then the projective curves $\tilde{C}_{cl}$ and $\tilde{C}'_{cl}$ have an additional $\overline{K}$-rational singular point $[0:0:1:0]$.		
	
	\item[d)] The normalisation $\overline{C}$ has genus $3$ and the rational degree $2$ map $\tilde{\eta}:\tilde{C}\rightarrow\tilde{\Delta}$ extends to a degree $2$ morphism $\eta:\overline{C}\rightarrow\overline{\Delta}$ which is ramified only at the four points $\infty_{1,\pm}$, $\infty_{2,\pm}$ which in the normalisation $\overline{\Delta}\rightarrow\tilde{\Delta}_{cl}$ map to $\infty_1$, $\infty_2$, respectively. The preimage $\eta^{-1}(\{\infty_{1,\pm},\infty_{2,\pm}\})$ maps through $\psi$ to the point $[0:0:1:0]$. The field of definition of the points $\infty_{i,\pm}$ and of their preimages is the field $K(\sqrt{d})$ where $d=\mu_{1,3}^2-4\mu_{0,3}\mu_{2,3}$.
	\end{enumerate}
\end{lemma}
\begin{proof}
	A rational map $\tilde{\eta}$ is clearly dominant, hence it induces an injection of the function fields $\overline{K}(\tilde{\Delta})\subset \overline{K}(\tilde{C})$ and hence we obtain a dominant map $\eta:\overline{C}\rightarrow\overline{\Delta}$ which agrees on some open subset with $\tilde{\eta}$, cf. \cite[Cor. I.6.12]{Hartshorne}.
	
	By \cite[Prop. II.6.8]{Hartshorne} it follows that $\eta$ is surjective, i.e., $\eta(\overline{C})=\overline{\Delta}$. The~restriction of $\tilde{\eta}$ to the non-singular open subset $U\subset \tilde{C}$ is unramified at each closed point. This follows from a direct computation. 
	
	Assume without loss of generality that $a_3\neq 0$. 
	
	Suppose that $P$ is separable.
	We then work with the singular model $\tilde{C}$ defined in \eqref{eq:CC1} and its projective closure $\tilde{C}_{cl}$. The affine patch $\tilde{C}$ has exactly six singular nodal points and the tangent cone at the point $(x_1,x_2,0)$ is $K$-isomorphic to $x^2-d y^2$ for $d=Q(x_1)Q(x_2)$. Point $(x_1,x_2,0)$ maps to $(x_1,x_2)$ which by Lemma \ref{lem:Delta_sing} is non-singular. Hence the map $\eta$ is unramified above these points since the degree of $\eta$ is $2$.
	
	Suppose that $P$ is not separable. If $P$ had a triple root, then curve $\tilde{\Delta}$ would have an additional singular point, hence its geometric genus would not be $1$.
	Suppose that $P$ has a double root. The extended Euclidean algorithm implies that the greatest common divisor of $P$ and $P'$ is polynomial of degree $1$ defined over $K$. It~follows that $P(x)=c(x-a)^2(x-b)$ where $a,b,c\in K$. Then $\tilde{C}$ has the model 
    \[\widetilde{\Delta}(x_1,x_2)=0,\ c^2(x_1-a)^2(x_2-a)^2 (x_1-b)(x_2-b)=y^2.\]
    The~birational map $(x_1,x_2,y)\mapsto (x_1,x_2,y/((x_1-a)(x_2-a)))$ changes the model $\tilde{C}$ to a new model which has only two nodal singularities at the points $(a,b,0)$ and $(b,a,0)$ with a tangent cone $K$-isomorphic to $x^2-d y^2$ for $d=Q(x_1)Q(x_2)$. Above the point $(a,a)$ on $\tilde{\Delta}$ we have two smooth points $(a,a,\pm(a-b))$. 
	
	Now, without additional assumption on the separability of $P$, we consider the ramification above the points $\infty_{1,\pm}$ and $\infty_{2,\pm}$ which map to $\infty_1=[1:0:0]$ and $\infty_2=[0:1:0]$ in the normalization map $\psi$. 

Finally, we will prove that the non-singular model $\overline{C}$ has four points $\infty_{C,1,\pm}$, $\infty_{C,2,\pm}$ which are obtained in the resolution of the singular point $[0:0:1:0]$ of $\widetilde{C}$. We are using the notations from the proof of Proposition \ref{prop:D_C_generic_implies_C_irr}.  The morphism $\eta$ maps $\infty_{C,i,\pm}$ onto $\infty_{i,\pm}$ and all of these points are defined over $K(\sqrt{d})$. Since the singularity $[0:0:1:0]$ is rather complicated we resolve the singularities in the birationally equivalent planar model $C_{pl}$ of $\tilde{C}$. Let $\kappa:C_{pl}\rightarrow \widetilde{C}$ be given by

\[\kappa(x_1,y)=\left(x_1,g\left(x_1,\frac{y P(x_1)}{T(x_1)}\right),\frac{yP(x_1)}{T(x_1)}\right)\]
where
\[C_{pl}:AT(x_1)^2 + \tilde{S}(x_1)y^2 + T(x_1)y^4=0.\]
In the projective model
\[A\widetilde{T}(x_1,z)^2 z^2 + \widetilde{S}(x_1,z)y^2 z + \widetilde{T}(x_1,z)y^4=0\]
the mapping is
\begin{equation}\label{eq:phi_map}
\kappa(x_1:y:z)=\left(x_1 ( \gamma x_1+\delta z)z:\left(y^2+\alpha x_1 z+\beta z^2\right)z:\frac{y \widetilde{P}(x_1,z)}{\widetilde{T}(x_1,z)} ( \gamma x_1+\delta z):z^2 ( \gamma x_1+\delta z)\right).
\end{equation}

We consider the points $[1:0:0]$, $ [0:1:0]$ or $[\alpha:0:1]$ for $T(\alpha)=0$ of $C_{pl}$ since only these points, or their blowups, can potentially map to $[0:0:1:0]$ on $\widetilde{C}$.
We will compute the image of these points under $\kappa$ and their composition with $\eta$.

\textbf{Claim 1:} The point $[0:1:0]$ on $C_{pl}$ is a node with tangents defined over $K(\sqrt{d})$.
The tangent cone at $[0:1:0]$ is $\widetilde{T}(x_1,z)=0$. The discriminant of $T$ equals $d$, so the slopes split over $K(\sqrt{d})$.

Since $y=1$ it is clear that the tangent cone has equation $\widetilde{T}(x_1,z)=0$. Due to our assumptions on $\widetilde{C}$ the discriminant of $T$ is non-zero, hence the singularity is an ordinary double point (a node).

\textbf{Claim 2:} Consider the blowup map $Bl_{[0:1:0]}: C_{pl}^{[0:1:0]}\rightarrow C_{pl}$. The preimage of $[0:1:0]$ contains two points defined over $K(\sqrt{d})$. Each of these two points maps via composition with $\kappa$ to $[0:0:1:0]$. Moreover, the composition with $\eta$ maps these points to the preimages of the blowup at $[0:1:0]$ of $\widetilde{\Delta}$.
In local coordinates we have $x_1=tz$, hence 
\[\kappa(t z:1:z)=\left(tz^2( \gamma tz+\delta z):\left(1+\alpha tz^2+\beta z^2\right)z:\frac{ \widetilde{P}(tz,z)}{\widetilde{T}(tz,z)} ( \gamma tz+\delta z):z^2 ( \gamma tz+\delta z)\right).\]

But $\widetilde{T}(tz,z) = T(t) z^2$ and the equation of $C_{pl}^{[0:1:0]}$ in the coordinates $(t,z)$ is $T(t)+z^2 F(t,z)=0$, hence $\widetilde{T}(tz,z)=-z^4 F(t,z)$. The points above $[0:1:0]$ in the blowup are given by $(t,z)$ where $z=0$ and $t$ satisfies $T(t)=0$. 
Observe that $Res_{t}(P(t),T(t)) = a_3^2 Res_{t}(P(t),Q(t))$, hence $P(t)$ does not vanish over $T(t)=0$. Moreover, $Res_{t}(T(t),\gamma t+\delta) = \mu_{2,3}^2 Res_{t}(P(t),Q(t))$, so $\gamma t+\delta$ does not vanish when $T(t)=0$.  Therefore, this demonstrates that the blowup points are mapped precisely to $[0:0:1:0]$.
The composition of $\kappa(tz:1:z)$ with $\eta$ results in the mapping
\[\eta(\kappa(t z:1:z))=\left(tz ( \gamma tz+\delta z)z\widetilde{T}(tz,z):\left(1+\alpha tz^2+\beta z^2\right)z\widetilde{T}(tz,z):z^2 ( \gamma tz+\delta z)\widetilde{T}(tz,z)\right)\]
which simplifies to
\[\eta(\kappa(t z:1:z))=\left(tz ( \gamma tz+\delta z):\left(1+\alpha tz^2+\beta z^2\right):z ( \gamma tz+\delta z)\right).\]
So the image of the blown-up points under composition with $\eta$ is $[0:1:0]$. Finally, blowing up the point $[0:1:0]$ on $\widetilde{\Delta}$ we see that the blown-up points $(t,0)$ with $T(t)=0$ map to points $(t,0)$ of the blowup of $\widetilde{\Delta}$.

\textbf{Claim 3:} Let $\gamma\neq 0$. There is only one point in blowup of $[1:0:0]$ on $C_{pl}$. The image of this point under $\kappa$ is $[0:0:1:0]$. The blowup point is a node. The points obtained after the second blowup map under composition with $\eta$ to the points in the blowup of $[1:0:0]$ of $\widetilde{\Delta}$.

We have that


\[\kappa(1:y:yt)=\left( ( \gamma +\delta yt)t:\left(y^2+\alpha yt+\beta (yt)^2\right)t:\frac{\widetilde{P}(1,yt)}{\widetilde{T}(1,yt)} ( \gamma +\delta yt):t^2 ( \gamma +\delta yt)\right)\]
with the equation of $C_{pl}$ in $(y,t)$ of the form
\[A\widetilde{T}(1,yt)^2 t^2 + \widetilde{S}(1,yt)y t + \widetilde{T}(1,yt)y^2=0.\]
So for $y=0$ we have $t=0$ since $A\neq 0$ and $\widetilde{T}(1,0)\neq 0$. Hence, the point in the blowup maps to $[0:0:\frac{a_3 \gamma}{-\mu_{2,3}}:0]$ which equals $[0:0:1:0]$ under our assumptions.

The blowup point is a node with tangent cone $A\mu_{2,3}^2 t^2+s_3 y t-\mu_{2,3} y^2=0$ and discriminant of the quadric equal to $\gamma^2 d$, hence the slopes are defined over $K(\sqrt{d})$.
Composition with $\eta$ maps the point $(y,t)$ to
\[\eta(\kappa(1:y:yt))=\left( ( \gamma +\delta yt):\left(y^2+\alpha yt+\beta (yt)^2\right):t ( \gamma +\delta yt)\right).\]
In particular, the point $(y,t)=(0,0)$ maps to $[1:0:0]$.

Blowup of $(y,t)=(0,0)$ gives the new equation in $y=y_1, t=y_1 t_1$ which is
\[A\widetilde{T}(1,y_1^2 t_1)^2 t_1^2 + \widetilde{S}(1,y_1^2 t_1) t_1 + \widetilde{T}(1,y_1^2 t_1)=0.\]
For $y_1=0$ we have a quadratic equation for $t_1$
\[A \mu_{2,3}^2 t_1^2+s_3 t_1-\mu_{2,3}=0.\]

Hence, the points above $[1:0:0]$ in the blowup of $\widetilde{\Delta}$ are in correspondence with the points $(0,t_1)$.

\textbf{Claim 4:} Let $\gamma=0$. There is only one point $\mathcal{Q}_{1}$ in the blowup of $[1:0:0]$ on $C_{pl}$. The point $\mathcal{Q}_{1}$ is singular non-ordinary and it maps under $\kappa$ to $[0:0:1:0]$. The blowup at $\mathcal{Q}_{1}$ has one preimage $\mathcal{Q}_{2}$ which is again singular non-ordinary. The composition of $\kappa$ with $\eta$ maps $\mathcal{Q}_{2}$ to $[1:0:0]$ on $\widetilde{\Delta}$. The next blowup at $\mathcal{Q}_{2}$ produces a single singular ordinary point $\mathcal{Q}_{3}$ which has tangent cone which splits over $K(\sqrt{d})$. The points in the blowup of $\mathcal{Q}_{3}$ map under the composition with $\kappa$ and $\eta$ to the blowup points above $[1:0:0]$ on $\widetilde{\Delta}$.

We argue like in the proof of Proposition \ref{prop:D_C_generic_implies_C_irr} to show that if $\gamma=0$, then $\delta\neq 0$. Let $z=yt$, then


\[\kappa(1:y:yt)=\left(  \delta t:\left(y+\alpha t+\beta yt^2\right):\frac{\widetilde{P}(1,yt)}{\widetilde{T}(1,yt)} \delta :\delta t^2  \right)\]
so the image of the blowup point is $[0:0:\frac{a_3 \delta}{-\mu_{2,3}}:0]$ which equals $[0:0:1:0]$ under our assumptions.
For $y=y_1$ and $t=y_1 t_1$ we obtain
\[\kappa(1:y_1:y_1^2 t_1)=\left(  \delta y_1 t_1:\left(y_1+\alpha y_1 t_1+\beta y_1(y_1 t_1)^2\right):\frac{\widetilde{P}(1,y_1^2 t_1)}{\widetilde{T}(1,y_1^2 t_1)} \delta :\delta (y_1 t_1)^2  \right)\]

thus

\[\eta(\kappa(1:y_1:y_1^2 t_1))=\left(  \delta t_1:1+\alpha  t_1+\beta (y_1 t_1)^2:\delta y_1  t_1^2  \right).\]

The composition with $\eta$ is well-defined and maps $y_1=0, t_1= -s_3/(2 A \mu_{2,3}^2)=-1/\alpha$ to the point $(1:0:0)$. The point $(y_1,t_1)=(0,-1/\alpha)$ is a cusp, so we blow it up again with $t_1-(-1/\alpha)=y_2 t_2$, $y_1=y_2$ obtaining


\[\eta(\kappa(1:y_2:y_2^2 (-1/\alpha+ y_2 t_2)))=\left(  \delta (-1/\alpha+ y_2 t_2):\alpha y_2 t_2+\beta (y_2 (-1/\alpha+ y_2 t_2))^2:\delta y_2  (-1/\alpha+ y_2 t_2)^2  \right).\]

Finally, we compute that the tangent cone at $(y_2,t_2)=(0,0)$ has separate slopes which split over $K(\sqrt{d})$ and the blowup points correspond to the blowup points above $[1:0:0]$ of $\widetilde{\Delta}$, similarly to the proof of Claim 3.

\textbf{Claim 5:} The points above $[\alpha:0:1]$ for $T(\alpha)=0$ in the respective blowups map to $[\alpha:*:*:1]$ in $\widetilde{C}$.

We check that each point $[\alpha:0:1]$ is a node and in the local coordinates we have $y=x_1' t$, $x_1=x_1'+\alpha$. The points above $[\alpha:0:1]$ in the blowup map to $[\alpha:*:*:1]$ due to \eqref{eq:phi_map}.
Hence, we conclude that each point $\infty_{i,\pm}$ is ramified of degree $2$ under $\eta:\overline{C}\rightarrow\overline{\Delta}$.
Given that these are the only ramified points, the Riemann-Hurwitz formula directly implies that the genus of $\overline{C}$ is $3$.
\end{proof}
Now we are ready to prove that $K$-typical is equivalent to $K$-generic.
\begin{proposition}\label{prop:typic_imp_generic}
	Suppose that $(\widetilde{\Delta},\widetilde{C})$ is $K$-typical. Then $\mu_{2,3}\neq 0$ and $\mu_{1,3}^2-4\mu_{0,3}\mu_{2,3}\neq 0$. In particular, the pair $(\widetilde{\Delta},\widetilde{C})$ is $K$-generic.
\end{proposition}
\begin{proof}
    Let $\mu=\mu_{1,3}^2-4\mu_{0,3}\mu_{2,3}$.
	If $\mu_{2,3}=0$, then the curve $\tilde{\Delta}$ is cubic. It must be smooth, since otherwise it would not be of genus $1$ by the Pl\"{u}cker formula. Hence $\mu_{1,3}\neq 0$ and $\mu=\mu_{1,3}^2-4\mu_{0,3}\mu_{2,3}\neq 0$. The projective closure $\tilde{\Delta}_{cl}$ has three points at infinity $\infty_1=[1:0:0]$, $\infty_2=[0:1;0]$ and $\infty_3=[1:-1:0]$. It follows from the assumptions above that each of these points is non-singular, hence $\tilde{\Delta}_{cl}=\overline{\Delta}$. 
	An argument similar to that in Lemma \ref{lem:C_sing}(d) shows that the points $\infty_1$ and $\infty_2$ are ramified and the point $\infty_3$ is unramified. In fact, the initial blowup of the point $[0:0:1:0]$ on $\tilde{C}_{cl}$ produces three points in the preimage, $\infty_1'$, $\infty_2'$ which are non-singular and the point $\infty_{3}'$ which is a non-ordinary double point. The next blowup of this point produces two preimages $\infty_{3,\pm}'$ which in the cover $\overline{C}\rightarrow\overline{\Delta}$ map to $\infty_{3}$. Hence, the Riemann-Hurwitz formula implies that the geometric genus of $\tilde{C}$ is $2$, which is a contradiction.
	
	If $\mu_{2,3}\neq 0$, then the curve $\tilde{\Delta}$ has degree $4$ and has two double points at infinity $[0:1:0]$ and $[1:0:0]$. 
	
	We may replace $P(x)$ with $x(x-1)(x-a)$, without loss of generality, as the argument remains consistent across $\overline{K}$-isomorphic models of $\tilde{C}$. For this model we have $\mu=(b_{1} - a b_{3})^2 - 4 b_{0} (b_{2} + (1 + a) b_{3})$. Assume $\mu=0$.
 
	In the case where $b_3=0$, it follows that $b_2\neq 0$. Given $\mu=0$, the resulting conditions simplify to $b_0=\frac{b_1^2}{4b_2}$ and $Q(x)=\frac{(b_1 + 2 b_2 x)^2}{4b_2}$. The model $\tilde{C}'$ of $\overline{C}$ is reducible, hence $\tilde{C}$ is not $K$-typical.
	
	We assume from now on that $b_3=1$.
	It follows from the proof of Proposition \ref{prop:D_C_generic_implies_C_irr} that the planar model of $\tilde{C}$ is reducible only if $\epsilon\neq 0$. 
	
	If $a=0$ and $\mu=0$, then $b_{1}^2 - 4 b_{0} (1 + b_{2})=0$. We have two possibilities. First, when $b_1=-2b_0$, which leads to $\tilde{\Delta}(x_1,x_2)=b_{0}(x_1-1) (x_2-1) (x_1 x_2 - x_1 - x_2)$ being reducible. Second, when $b_1\neq-2b_0$. We want to prove that the polynomial $B(x,y)$ from Proposition \ref{prop:D_C_generic_implies_C_irr} is reducible in this case. The number $A$ from \eqref{eq:A_formula} equals
	\[A=Res_{x}(P(x),Q(x))=b_{0}^2(1+b_{0}+b_{1}+b_{2}).\]
    Since $b_2=(-4b_{0}+b_{1}^2)/(4b_{0})$, it follows that $A=\frac{b_0}{4} (2 b_0 + b_1)^2$.
    It follows from the proof of Proposition \ref{prop:D_C_generic_implies_C_irr} that if the polynomial $B(x,y)$ factors, it can only factor into two quadratic polynomials in $y$, i.e., $A(1-\epsilon y+\phi y^2)(1+\epsilon y+\phi y^2)$ where
	\[\phi=\frac{(2 b_0 + b_1 x)^3}{4 b_0^2 (2 b_0 + b_1)}\]
	and
	\[\epsilon=-\frac{b_0 (4 b_0 + b_1) + b_1 (3 b_0 + b_1) x}{b_0 (2 b_0 + b_1)}.\]
	
	Since $\mu=0$ it follows that when $b_0 \left(-a+2 b_0+b_1\right) \left(a \left(b_1-a\right)+2 b_0\right)=0$, the algebraic set $\tilde{\Delta}(x,y)=0$ is not irreducible and is irreducible otherwise. 
    \begin{itemize}
        \item When $\mu=0,\ b_{0}=0$ we have $b_{1}=a$ and so $\tilde{\Delta}(x,y) = x_1 x_2 (x_1 x_2-a)(1+a+b_{2})$.
        \item When $\mu=0,\ b_{1}=a-2b_{0}$, we have $\tilde{\Delta}(x,y) = b_{0}(x_1-1)(x_2-1)(x_1 x_2-x_1-x_2+a)$.
        \item When $\mu=0, \ b_{0} = a(a-b_{1})/2$, we have $\tilde{\Delta}(x,y) = \frac{a-b_{1}}{2a} (x_1-a)(x_2-a)(x_1 x_2-a(x_1+x_2)+a)$.
        
    \end{itemize}
 
    Hence, the remaining cases are covered by the following substitution
	\[b_2=\frac{\left(a-b_1\right){}^2}{4 b_0}-a-1\]
    obtained from $\mu=0$.
    The coefficients $A=\frac{\left(-a+2 b_0+b_1\right){}^2 \left(a \left(b_1-a\right)+2 b_0\right){}^2}{16 b_0}$,
	\[\phi = -\frac{\left(x \left(a-b_1\right)-2 b_0\right){}^3}{2 b_0 \left(-a+2 b_0+b_1\right) \left(a \left(b_1-a\right)+2 b_0\right)}\]
	and
	\[\epsilon=\frac{x \left(b_1-a\right)}{2 b_0}+2 b_0 \left(\frac{x-1}{a-2 b_0-b_1}+\frac{x-a}{a \left(a \left(a-b_1\right)-2 b_0\right)}\right)+\left(\frac{1}{a}+1\right) x\]
	provide the factorization of the polynomial $B(x,y)$.
	Hence $\mu=\mu_{1,3}^2-4\mu_{0,3}\mu_{2,3}\neq 0$. 
	The last claim follows from the assumption about $\widetilde{\Delta}$ which has geometric genus $1$ and degree $4$ so there are two nodal singularities at infinity and by Pl\"{u}cker formula there are no singularities at the affine part, proving the $K$-genericity of $(\widetilde{\Delta},\widetilde{C})$.
\end{proof}

\begin{corollary}\label{cor:typical_vs_generic}
	A pair of curves $(\widetilde{\Delta},\widetilde{C})$ is $K$-generic if and only if it is $K$-typical.
\end{corollary}
\begin{remark}
 Notice that the curve $\widetilde{\Delta}$ can be of genus $1$ while not being $K$-generic. This happens when $\mu_{2,3}=0$ or $\mu_{1,3}^2-4\mu_{0,3}\mu_{2,3}=0$.  In such a case, the curve $\widetilde{C}$ is typically a union of two genus $1$ curves and the restriction of the projection $\pi_1$ to each component is a degree $1$ map or has geometric genus $2$ and the restriction of $\pi_1$ is a degree $2$ map, cf. Table \ref{tab:non_generic_cases}.
\end{remark}

\subsection{Applications}
In order to see what is the contribution of $\#\tilde{C}(\Fq) - \#\tilde{\Delta}(\Fq)$ to the bias, we need to compare $\#\tilde{C}(\Fq)$ and $\#\tilde{\Delta}(\Fq)$ to $\#\overline{C}(\Fq)$ and $\#\overline{\Delta}(\Fq)$, where $\overline{C}$ and $\overline{\Delta}$ are non-singular and projective models of $\tilde{C}$ and $\tilde{\Delta}$ respectively.

Let  $S_{\overline{\Delta},\infty}$ and $S_{\overline{C},\infty}$ denote the closed subscheme defined over $K$ which are a union of  closed points $\{\infty_{1,\pm},\infty_{2,\pm}\}$ and $\eta^{-1}(\{\infty_{1,\pm},\infty_{2,\pm}\})$, respectively. 

Let $\mathcal{G}$ denote the union of the closed points in $\overline{C}$ which map via $\psi$ to the singular points of either $\tilde{C}$ or $\tilde{C}'$. The scheme $\mathcal{G}$ is defined over $K$.
\begin{corollary}\label{cor:singular}
	Let $q$ be an odd prime power. Let $K=\mathbb{F}_{q}$ and assume that $(\tilde{\Delta},\tilde{C})$ is $K$-typical.
	We have $|S_{\overline{\Delta},\infty}(\mathbb{F}_{q})|=|S_{\overline{C},\infty}(\mathbb{F}_{q})|$. 
	
	Moreover, when $a_3\neq 0$ and $x\in\mathcal{G}(\mathbb{F}_{q})$ is such that there exists two elements $x_1,x_2\in\mathbb{F}_{q}$ with $\psi(x)=(x_1,x_2,0)$, the set $\eta^{-1}(\{x\}) (\mathbb{F}_{q})$ has exactly two elements if $Q(x_1)Q(x_2)$ is a non-zero square in $\mathbb{F}_{q}$. Otherwise, if $Q(x_1)Q(x_2)$ is not a square, the set $\eta^{-1}(\{x\}) (\mathbb{F}_{q})$  is empty.
\end{corollary}
\begin{proof}
	This follows directly from the Lemmas \ref{lem:Delta_sing} and \ref{lem:C_sing}.
\end{proof}

\begin{corollary}\label{cor:final}
	Let $q$ be an odd prime power. Let $K=\mathbb{F}_{q}$ and assume that $(\tilde{\Delta},\tilde{C})$ is $K$-typical and $a_3\neq 0$. We have that
\begin{align*}
\tilde{M}_{2,q}(\mathcal{E})/q&=\#\overline{C}(\Fq)-\#\overline{\Delta}(\Fq)+q-\#S(\Fq)
\end{align*}

\end{corollary}

\begin{proof}
It follows from Proposition \ref{prop:Curve_defect} and Theorem \ref{thm:Second_moment_curve_formula} that
\begin{align*}
\frac{\tilde{M}_{2,q}(\mathcal{E})}{q} &= \# \tilde{C}(\Fq)-\#\tilde\Delta(\Fq)+q-\#S(\Fq)-\#P(\Fq)\\
&+\#(P \cap S)(\Fq)+\left[ \sum_{P(x) \equiv 0} \phi_q(Q(x))\right]^2.
\end{align*}
If $P$ is separable, it follows from Lemma \ref{lem:C_sing} and Corollary \ref{cor:singular} that
\begin{align*}
\#\tilde{C}(\Fq)-\#\tilde\Delta(\Fq)&=\#\overline{C}(\Fq)-\#\overline\Delta(\Fq)+\#P(\Fq)\left(\#P(\Fq)-1\right)\\
&-\sum_{\substack{x_1,x_2\in \Fq\\P(x_1)=P(x_2)=0\\x_1 \ne x_2}}\left(\phi_q(Q(x_1)Q(x_2))+1\right).
\end{align*}

Hence,
\begin{align*}
\frac{\tilde{M}_{2,q}(\mathcal{E})}{q} &= \# \overline{C}(\Fq)-\#\overline\Delta(\Fq)+q-\#S(\Fq)-\#P(\Fq)\\
&+\#(P\cap S)(\Fq)+\sum_{\substack{x\in \Fq\\P(x)=0}}\phi_q(Q(x))^2.
\end{align*}
The claim follows since
$$\sum_{\substack{x\in \Fq\\P(x)=0}}\phi_q(Q(x))^2=\#P(\Fq)-\#(P\cap Q)(\Fq)$$
and the polynomials $P$ and $Q$ have no common roots since the curves are $K$-typical, hence $K$-generic, hence $\#(P\cap Q)(\Fq)=0$ and $\#(P\cap S)(\Fq)=0$.

When $P$ is non-separable it follows from Lemma \ref{lem:C_sing} that $\#(P\cap Q)(\Fq)=0$ and $\#(P\cap S)(\Fq)=1$. The remaining modifications are easy to replicate from the previous case.
\end{proof}

Suppose that $\tilde{C}$ is geometrically irreducible and geometrically reduced. We denote by $\overline{C}$ nonsingular projective curve birationally equivalent to the projective closure of $\tilde{C}$. Let $\tau_1, \tau_2$ and $\tau_3$ be three involutions of $\overline{C}$ whose restriction to $\tilde{C}$ are equal to $(x_1, x_2, y)\mapsto (x_1,x_2,-y)$, $(x_1, x_2, y)\mapsto (x_2,x_1,-y)$ and $(x_1, x_2, y)\mapsto (x_2,x_1,y)$ respectively. Denote by $\phi_i:\overline{C} \rightarrow C_i$ the natural map from $\overline{C}$ to the non-singular projective quotients of $\overline{C}$ by $\tau_i$. Note that $C_1=:\overline{\Delta}$ is the nonsingular projective closure of $\tilde{\Delta}$. Moreover, denote by $\phi_4:\overline{C} \rightarrow C_4$ the map to the nonsingular projective quotient of $\overline{C}$ by the group $G$ generated by involutions $\tau_i$ or equivalently the quotient of $\overline{\Delta}$ by an automorphism of $\overline{\Delta}$ induced by $(x_1,x_2)\mapsto (x_2,x_1)$. 
\begin{proposition}\label{prop:quotient} 
	For a prime power $q$ and $\tilde{C}$ geometrically irreducible and geometrically reduced over $\Fq$ we have
	$$ \#\overline{C}(\Fq)+2\#C_4(\Fq) = \#C_1(\Fq)+\#C_2(\Fq)+\#C_3(\Fq).$$
\end{proposition}
\begin{proof}
	Denote by $\sigma$ the generator of $\Gal(\mathbb{F}_{q^2}/\Fq)$.	
	Since $\overline{C}(\overline{\Fq})$ is a disjoint union of orbits $\mathcal{O}$ of $G$ acting on $\overline{C}(\overline{\Fq})$. We denote by $\mathcal{O}(\F_q)$ the set of $\F_{q}$-rational points of $\mathcal{O}$. It is enough to prove that for each orbit $\mathcal{O}$ we have
	\begin{equation}\label{eq:orbit}
	\#\OO(\Fq)+2\#\phi_4(\OO)(\Fq) = \#\phi_1(\OO)(\Fq)+\#\phi_2(\OO)(\Fq)+\#\phi_3(\OO)(\Fq).
	\end{equation}
	The group $G$ has $4$ elements, so by the orbit-stabilizer theorem the size of the orbit $\mathcal{O}$ is a divisor of $4$.
	\begin{enumerate}
		\item [a)] Case $\#\OO=1$. Here \eqref{eq:orbit} follows immediately.
		\item [b)] Case $\#\OO=2$. There is a $\tau=\tau_i$ such that every element of $\#\OO$ is fixed by $\tau$. If $\#\OO(\Fq)=2$, then the left-hand side of \eqref{eq:orbit} is equal to $2+2$, while the right-hand side is equal to $2+1+1$ (since $\#\phi_i(\OO)=2$). If $\#\OO(\Fq)=0$, then we have two possibilities. First, if $\#\phi_4(\OO)(\Fq)=0$, then the right-hand side of \eqref{eq:orbit} is zero since there are natural $\Fq$-rational quotient maps from $C_j \rightarrow C_4$. If $\#\phi_4(\OO)(\Fq)=1$, then for every $P \in \OO$ and any involution $\tau'$ different than $\tau$ we have $P^\sigma = \tau' P$ where $\sigma$ is a generator of $\Gal(\mathbb{F}_{q^2}/\Fq)$. In particular, $\#\phi_1(\OO)(\Fq)+\#\phi_2(\OO)(\Fq)+\#\phi_3(\OO)(\Fq)=2$ as required.
		\item[c)] Case $\#\OO=4$. If $\#\OO(\Fq)=4$, then $\#\phi_4(\OO)(\Fq)=1$ and $ \#\phi_i(\OO)(\Fq)=2$ for $i=1,2,3$ (since $G$ acts faithfully on $\OO$) hence the claim follows. If~$\#\OO(\Fq)=0$ (since involutions are $\Fq$-rational this is the only other option), we have two possibilities. First, if $\#\phi_4(\OO)(\Fq)=0$, then the claim follows as above. If $\#\phi_4(\OO)(\Fq)=1$, then for every $P \in \OO$ there is an involution $\tau$ such that $P^\sigma = \tau P$ (note that in this case $P$ can not be defined over quartic extension of $\Fq$). The claim now follows as in b).
	\end{enumerate}
\end{proof}

\begin{remark}
Proposition \ref{prop:quotient} is a finite field analogue of the theorem \cite[Thm. 5.9]{Accola} applied to the automorphism group $\mathbb{Z}/2\oplus\mathbb{Z}/2$ acting on $\overline{C}$.  It follows from \cite[5.10]{Accola} that
\begin{equation}\label{eq:accola_cor}
	g(\overline{C})+2g(C_4)=g(C_1)+g(C_2)+g(C_3)
\end{equation}
when the field characteristic of the curves $\overline{C}$, $C_i$ is odd or zero.
\end{remark}

\begin{proposition}\label{prop:C_4_smooth_conic}
	Let $K$ denote a field such that $\charac K \neq 2$ and let $(\tilde{\Delta},\tilde{C})$ be a  $K$-typical pair.
	The quotient $C_4$ of $\overline{C}$ by the group $\langle \tau_1,\tau_2\rangle$ is isomorphic to a smooth conic, hence $g(C_4)=0$.
\end{proposition}
\begin{proof}
	Let $s_1=x_1+x_2, s_2=x_1\cdot x_2$ and $\eta=y^2$. These polynomials are a complete set of primary invariants of the group $\langle \tau_1,\tau_2\rangle$ acting on the polynomial ring $K[x_1,x_2,y]$. The curve $C_4$ has a model
	\[\mu _{1,0}+s_1^2 \mu _{3,0}+s_1 \mu _{2,0}+s_2 s_1 \mu _{3,1}+s_2 \left(2 \mu _{0,3}+\mu _{2,1}+\mu _{3,0}\right)+s_2^2 \mu _{3,2}=0,\quad R(s_1,s_2)=\eta\]
	where $R(s_1,s_2)=P(x_1)P(x_2)$. We eliminate the second equation. The discriminant of the first equation is $-16\mu_{2,3} Res_{x}(P(x),Q(x))\neq 0$ which is non-vanishing due to Corollary \ref{cor:typical_vs_generic}. Hence, the conic $C_4$ in the $s_1,s_2$-plane is not a union of two lines, hence geometrically irreducible and reduced.
\end{proof}

\begin{remark}
	A given point $(x_0,x_0)$, $x_0\in \overline{K}$, is singular on $\tilde{\Delta}$ if and only if it is a common root of the polynomials $\tilde{\Delta}(x,x)$ and $\delta(x)=\delta_1(x,x)=\delta_2(x,x)$ from Section \ref{sec:generic}.
	In fact, the resultant $\Res_{x}(\tilde{\Delta}(x,x),\delta(x))$ equals $-1/16 \cdot\mu_{2,3}\cdot \textrm{disc} \tilde{\Delta}(x,x)$ and $\mu_{2,3}\neq 0$ in the $K$-generic context.
	Hence, a point of the form $(x_0,x_0)\in \tilde{\Delta}(\overline{K})$ is singular if and only if  $x_0$ is a multiple root of $\tilde{\Delta}(x,x)$.
	Thus the polynomial $\tilde{\Delta}(x,x)$ is separable when the curves $\tilde{\Delta}$ and $\tilde{C}$ are $K$-generic.
\end{remark}

\begin{corollary}\label{prop:D}
	Let $K$ denote a field such that $\charac K \neq 2$.
	For a $K$-typical pair $(\tilde{\Delta},\tilde{C})$ we have that the genera of $\overline{C}, C_1=\overline{\Delta}$, $C_2$, $C_3$ and $C_4$ are equal to $3, 1, 2, 0$ and $0$, respectively. 
	Hence, the curve $\overline{C}$ is bielliptic and hyperelliptic.
	Let $q$ be an odd prime power. 
	When the pair $(\tilde{\Delta},\tilde{C})$ is $\mathbb{F}_q$-typical we have 
	\begin{equation}\label{eq:count_formula_C_delta}
		\#\overline{C}(\Fq)-\#\overline{\Delta}(\Fq)=\#C_2(\Fq)-(q+1).
	\end{equation}
	
\end{corollary}
\begin{proof}
	Each map $\phi_i:\overline{C}\rightarrow C_i$ is a finite separable map onto a smooth projective curve $C_i$. Since the curve $\overline{C}$ has genus $3$ by Lemma \ref{lem:C_sing}, it follows from the Riemann-Hurwitz formula \cite[IV Cor. 2.4]{Hartshorne} that
	\[2=g(C_i)+\frac{r}{4}\]
	where we denote by $r$ the degree of the ramification divisor. Thus, $r$ belongs to the~set $\{0,4,8\}$.
	
	Maps $\phi_{1},\phi_{2},\phi_{3}$ are covers of degree $2$ and their ramification locus corresponds to the number of fixed points of the involutions $\tau_1,\tau_2,\tau_3$, respectively.
	
	In Lemma \ref{lem:C_sing}, we already discussed the case of $\eta=\phi_1$, which is ramified at four points, proving that $r=4$ and $g(C_1)=1$. The fixed points of $\tau_1$ are the four points $\{\infty_{i,\pm}\}$ on $\tilde{C}$.
	
	Proposition \ref{prop:C_4_smooth_conic} proves that $g(C_4)=0$. Thus, it follows from \eqref{eq:accola_cor} that $2=g(C_2)+g(C_3)$. Moreover, we have that
	the four points $\{\infty_{i,\pm}\}$ on $\overline{C}$ are not fixed by $\tau_2$, since they are fixed by $\tau_1$. Indeed, if they were fixed by $\tau_1$ and $\tau_2$ the contribution to the ramification divisor of the map $\overline{C}\rightarrow C_4$ would be too large.
	
	We claim that the involution $\tau_2:(x_1,x_2,y)\mapsto (x_2,x_1,-y)$ has at most one fixed point on $\tilde{C}$ and at most two fixed points on $\overline{C}$. 
	An inclusion $(x,x,0)\in \tilde{C}(\overline{K})$ for a certain $x\in\overline{K}$ would imply that 
	\begin{equation}\label{eq:deg_condition_C}
	P(x)=\tilde{\Delta}(x,x)=0.
	\end{equation}
 
    When $P$ is separable, this is not possible due to $K$-genericity. Then, $\tau_2$ has no fixed points, hence $r=0$ and $g(C_2)=2$ and $g(C_3)=0$.
    
    Otherwise, when $P$ is not separable it follows from Lemma \ref{lem:C_sing} that $P$ has a double root $a\in K$ such that $\tilde{\Delta}(x,x)=(x-a)h(x)$ and $h(a)\neq 0$ which follows by $K$-genericity. Hence on $\overline{C}$ we have exactly two smooth points over the point $(a,a,0)$ on $\tilde{C}$. Since the points $\{\infty_{i,\pm}\}$ are not fixed by $\tau_2$, then $r\leq 2$, hence $r=0$ which concludes the genera computation.
	
	Since the curve $\overline{C}$ has a degree $2$ map onto a genus $1$ curve it is bielliptic and since it has also a degree $2$ cover onto a genus $0$ curve, it is hyperelliptic.

	The formula \eqref{eq:count_formula_C_delta} follows from Proposition \ref{prop:quotient} and the fact that both $C_3$ and $C_4$ are smooth projective of genus $0$. The equality $\# C_3(\mathbb{F}_q)=\#	C_4(\mathbb{F}_q)=q+1$ follows from the Hasse-Weil bound.
\end{proof}

\begin{remark}\label{rem:genus_2_curve_explicit}
	The Weierstrass points of $C_2$ are ramification points of the quotient map $C_2 \rightarrow C_4$ induced by the push-forward of the involution $\tau_3$ to $C_2$, i.e., the fixed points of that involution on $C_2$. These are the two points at infinity, and four affine points given by the roots of $\tilde{\Delta}(x,x)=0$. 
\end{remark}

\section{Proof of the Bias conjecture for typical curves} \label{sec:contribution}

To compute the bias and prove our main result in Theorem \ref{thm:main2}, we need to understand the average over primes of the remaining term from Corollary \ref{cor:final}, namely \( \#S(\mathbb{F}_p) \). To address this, we consider the following general situation.

Let $f(x)\in \ZZ[x]$ be an irreducible polynomial. For a prime $p$, we denote by $\#Z_{f}(\Fp)=\#\{\alpha \in \Fp: f(\alpha)=0\}$ the number of distinct zeros of the mod $p$ reduction of $f(x)$. 

\begin{proposition}\label{cor:average}
	Let $f(x)\in \ZZ[x]$ be an irreducible polynomial. We have that
	$$\lim_{x\rightarrow \infty} \frac{1}{\pi(x)}\sum_{p<x} \#Z_{f}(\Fp)=1,$$
	where $\pi(x)$ denotes the number of primes less than $x$.
\end{proposition}
\begin{proof}
	Denote by $\rho$ a Galois representation $\rho:\GalQ \rightarrow GL(V)$ which is a permutation representation of the absolute Galois group acting on the roots $\alpha_1, \ldots, \alpha_d \in \overline{\QQ}$ of $f(x)$. Its is a $d$-dimensional representation that factors through $G=\Gal(\QQ(\alpha_1, \ldots, \alpha_d)/\QQ)$, and since $f(x)$ is irreducible over $\QQ$ it  contains exactly one copy of one dimensional trivial representation (which is a restriction of $\rho$ to the subspace of $V$ generated by formal sum $\alpha_1+\cdots+\alpha_d$). Moreover, for all but finitely many primes $p$, we have that $\#Z_{f}(\Fp)=Trace(\rho(Frob_p))$, where $Frob_p\in G$ is a Frobenius element at $p$. By the Chebotarev density theorem, the elements $Frob_p$ are equidistributed among conjugacy classes of $G$. We denote by $\chi_\rho$ the character of $\rho$ (i.e., $\chi_\rho(g)=Trace(\rho(g))$ for all $g \in G$). We have that
	$$\lim_{x\rightarrow \infty} \frac{1}{\pi(x)}\sum_{p<x} \#Z_{f}(\Fp)(p)=\frac{1}{\#G}\sum_{g \in G}\chi_\rho(g)=\left<\chi_\rho,\chi_{triv} \right>,$$ 
	where $\chi_{triv}$ is the character of trivial representation. The claim follows from orthogonality relations for characters of irreducible representations and the fact that $\rho$ contains only one copy of the trivial representation.
\end{proof}

Now we are ready to prove our main result.

\begin{proof}[Proof of Theorem \ref{thm:main2}]
    There exists only finitely many primes $p$ such that there exists $u_0\in\mathbb{P}^{1}(\Fq)$, $p\mid q$, such that the fiber $\mathcal{E}_{u_0}$ is not geometrically irreducible. Suppose not, then there exists a parameter $u_0\in\mathbb{Q}$ such that $\mathcal{E}_{u_0}$ is a union of two lines (or a double line). This is only possible when $P(x)u_0+Q(x)$ is a polynomial of degree $2$ with a double root or $P(x)u_0+Q(x)=\lambda\in\mathbb{Q}$ or $P(x)$ is a zero polynomial. Each such case is excluded by the genericity assumption. So, without loss of generality, we can assume that $p$ is large enough so that each fiber $\mathcal{E}_{u_0}$ is geometrically irreducible.
	
	For a prime $p$ let $d_p=p+1-\#\overline{D}(\Fp)$. It follows from Proposition \ref{prop:quotient} and Corollary \ref{prop:D} that one has  $$\#\overline{C}(\Fp)-\#\overline{\Delta}(\Fp)=-d_p.$$ The Weil conjectures imply that for all but finitely many $p$ one has 
	$$f_4(p)=p^2,\quad  f_3(p)=-p \cdot d_p,\quad f_2(p)=-p\cdot\#S(\Fp)\textrm{ and }  f_j(p)=0 \textrm{ for } j<2.$$

It follows from the argument after \eqref{eq:second_moment}, applied to the genus two curve \( \overline{D} \), that the average of \( \frac{d_p}{\sqrt{p}} \) over the primes is equal to zero, hence the $f_{3}(p)$ term in the $p$ power expansion of $\tilde{M}_{2,p}(\mathcal{E}_{U})$ averages to zero. The average of $f_2(p)$ term equals $-m$ by Proposition \ref{cor:average} applied to the polynomial 
	$$S(x)=\tilde{\Delta}(x,x)=	\mu _{3,2} x^4 +2  \mu _{3,1}x^3+\left(\mu _{2,1}+3 \mu _{3,0}\right)x^2 +2\mu _{2,0} x +\mu _{1,0}$$
	where $\mu_{i,j}=a_i b_j -a_j b_i$.
	The expressions $\tilde{M}_{2,p}$ and $M_{2,p}$ differ by the number $a_{\infty,p}^2$. When the fiber $\mathcal{E}_{\infty}$ is an elliptic curve, then the argument preceding \eqref{eq:sec_moment} implies that the  sequence $\{a_{\infty,p}^2\}_{p}$ has average equal to $1$ and the term contributes to the average of $f_{2}(p)$. When the fiber is singular (but geometrically irreducible), then the contribution $a_{\infty,p}^2\in\{0,1\}$ and the average is again $0,1$, respectively and contributes to the the average of $f_0(p)$. Hence the average of $f_2(p)$ for $M_{2,p}$ is $-m-\delta$ where $\delta\in\{0,1\}$ and so the Bias conjecture holds.
	
	\textbf{Family with $m=-1$ or $m=-2$}:
 In the affine space $\mathbb{A}^8$ of $8$-tuples $(a_0,a_1,a_2,a_3,b_0,b_1,b_2,b_3)$ the subset of parameters for which $(\tilde{\Delta},\tilde{C})$ are $\mathbb{Q}$-typical contains a Zariski open subset of positive dimension and the polynomial $S(x)$ is irreducible over $\mathbb{Q}(a_0,a_1,a_2,a_3,b_0,b_1,b_2,b_3)$. 
 By the Hilbert irreducibility theorem, there are infinitely many cases where $S$ is irreducible over $\mathbb{Q}$. Two families $\mathcal{E}_U$ and $\mathcal{E}'_U$ are isomorphic only when the $j$-invariants are equal which is another closed condition. Finally, the vanishing of the discriminant $\Delta(\mathcal{E}_{\infty})$ of $\mathcal{E}_{\infty}$ is a closed condition which is satisfied for example when $a_3=0$. Hence, the bias is equal to $-1$ or $-2$ for infinitely many choices of parameters.
	
	Let $A$ denote the tuple $(a_0,a_1,a_2,a_3)$ and $B$ denote the tuple $(b_0,b_1,b_2,b_3)$.
	
	\textbf{Family with $m=-3$}: Consider $A=(0, s, -1 - s, 1)$ and 
	\begin{align*}
	B=\Bigl(&\frac{s+1}{s^2+3 s+4},b_1,-\frac{b_1 s^3+4 b_1 s^2+7 b_1 s+4 b_1-s^2-4 s-1}{s \left(s^2+3 s+4\right)},\\&\frac{b_1 s^2+3 b_1 s+4 b_1+s^2+s-1}{s \left(s^2+3 s+4\right)}\Bigr)
	\end{align*}
	where $s\in\mathbb{Q}\setminus\{0,-1,-2\}$. 
	For $s\notin\{-1,0\}$ the coefficients $\mu_{i,j}$ do not vanish. The~polynomials $P$, $Q$ have no common roots if $s\notin\{0,-1,-2\}$. Next, the expression $\mu_{1,3}^2-4\mu_{0,3}\mu_{2,3}=\frac{\left(s^2-5 s-5\right) \left(s^2+3 s+3\right)}{\left(s^2+3 s+4\right)^2}\neq 0$ and $$S(x)=\frac{\left(x^2+x+1\right) \left(s^2 x^2-3 s^2 x+s^2+3 s x^2-5 s x+s+4 x^2-2 x\right)}{s^2+3 s+4}.$$
	Polynomial $S(x)$ has exactly two irreducible factors over $\mathbb{Q}$ when $s\neq -1,-2$. It~follows from Definition \ref{def:K_generic_tuple} and Corollary \ref{cor:typical_vs_generic} that the tuple $(\tilde{\Delta},\tilde{C})$ is $\mathbb{Q}$-typical.  For a pair $(b_1,s)$ and $(b_1',s')$ the $j$-invariants of the corresponding $j$-invariants $j$, $j'$ of the families are equal if and only if $s'\in\{\frac{1}{1-s},1-s,\frac{1}{s},\frac{s-1}{s},s,\frac{s}{s-1}\}$. The discriminant $\Delta(\mathcal{E}_{\infty})=(s-1)^2 s^2$ is non-vanishing.
	
	\textbf{Family with $m=-4$}:
	Let $A=(0, 1, -2, 1)$ and $$B=\left(\frac{2 (s+4)^2}{25 s},b_1,-\frac{20 b_1 s+s^2+28 s+16}{10 s},\frac{20 b_1 s+s^2+38 s+16}{20 s}\right).$$ 
	For $s\in\mathbb{Q}\setminus\{0, -4, 6, 8/3, 8, 2\}$  the pair $(\tilde{\Delta},\tilde{C})$ is $\mathbb{Q}$-typical and the polynomial $$S(x)=-\frac{(x-2) (x-1) (s-10 x+4) (5 s x-2 s-8)}{50 s}$$ has four linear factors over $\mathbb{Q}$. Two pairs $(b_1,s)$ and $(b_1',s')$ generate the same $j$-invariant only if they are equal or $b_1=b_1'$ and $s'=\frac{16}{s}$. Finally, the discriminant $\Delta(\mathcal{E}_{\infty})$ vanishes.
	
	\textbf{Family with $m=-5$}:
	Let $A=(0, 37/16, -(53/16), 1)$ and $$B=\left(-\frac{576}{65},b_1,\frac{5744-3445 b_1}{2405},\frac{16 \left(65 b_1-63\right)}{2405}\right).$$
	For any $b_1\in\mathbb{Q}$ the pair $(\tilde{\Delta},\tilde{C})$ is $\mathbb{Q}$-typical, $S(x)=\frac{1}{65} (x-3) (x-2) (5 x+37) (13 x-6)$, the $j$-invariants are distinct for distinct $b_1$'s and $\Delta(\mathcal{E}_{\infty})=603729/65536$.
\end{proof}

\section{Non-typical cases}\label{sec:non_typical}

The non-typical choices of $(\tilde{\Delta},\tilde{C})$  can be characterised using the violation of the $K$-genericity condition. Assume in what follows that the field $K$ has characteristic $\charac K\neq 2,3$.

If the pair $P,Q$ of polynomials has at least one common root, we have the following types. Assume that the common factor of $P$ and $Q$ is denoted by $R$. If
\begin{itemize}
	\item $\deg R=1$, then $\tilde{\Delta}$ is a union of two lines and a conic and $\tilde{C}$ is a union of two doubles lines and a curve which (if geometrically irreducible) is of geometric genus at most $1$.
	\item $\deg R=2$ case is not considered because then the family $\mathcal{E}_{U}$ is a family of singular cubics.
	\item $\deg R=3$, then we have a family of twists of a single curve.
\end{itemize}

Suppose now that $\deg R=0$. We use the notation from Section \ref{sec:generic}. We characterize the remaining non-typical setups  of $\mathcal{T}=(a_0,a_1,a_2,a_3,b_0,b_1,b_2,b_3)\in K^8$ by the following conditions:
\[\mathcal{C}_{1}: \mathcal{T}\in \varrho(\mathcal{P}_1),\]
\[\mathcal{C}_{2}: \mathcal{T}\in \varrho(\mathcal{P}_2),\]
\[\mathcal{C}_{3}: \mathcal{T}\in \varrho(\mathcal{S}).\]
Notice that $\mathcal{C}_1$ is equivalent to $\mu_{2,3}=0$, $\mathcal{C}_2$ is equivalent to $\mu= \mu_{1,3}^2-4\mu_{0,3}\mu_{2,3}=0$ and the condition $\mathcal{C}_{3}$ implies that the following polynomial
\[\mu _{1,2}^3+27 \mu _{0,1} \mu _{1,3}^2+27 \left(\mu _{0,2}^2-\mu _{0,1} \mu _{1,2}\right) \mu _{2,3}-9 \mu _{0,3} \left(\mu _{1,2}^2+9 \mu _{0,1} \mu _{2,3}\right)\]
vanishes.

 In Table \ref{tab:non_generic_cases} we discuss a geometric situation in each of the $8$ vanishing setups for the triple $(\mathcal{C}_1,\mathcal{C}_2,\mathcal{C}_3)$ of conditions. The entries should be read in the following way. In line $i$ the conditions $(\mathcal{C}_1,\mathcal{C}_2,\mathcal{C}_3)$ being true or false determine a locus $\mathcal{M}$ of coefficients $\mathcal{T}\in\overline{K}^8$ such that the corresponding curves $\tilde{\Delta}$ and $\tilde{C}$ satisfy the properties described in the ''geometric setup'' column. The statement $g(X)$ for an $X$ means that away from some $\mathcal{T}$ in the appropriate Zariski closed subsets of $\overline{K}^{8}$ of positive codimension the given curve $X$ is geometrically irreducible and reduced and its genus $g(X)$ is given as in the entry. We explain in more detail below what happens in each case.
 
 We notice that the verification of the Bias conjecture in cases $2$-$8$ from Table \ref{tab:non_generic_cases} does not depend on the unproven cases of the Sato-Tate conjecture for curves. In cases $2,3,4$ we apply the formula \eqref{eq:first_moment} for genus $1$ curves. In cases $4,6,7$ and $8$ we only deal with genus $0$ curves and in the case $5$ the quotient curve $C_2$ has genus $1$ as well due to \eqref{eq:accola_cor}.
\begin{table}[htb]
	\begin{tabular}{c|c|c|c|c}
		& $\mathcal{C}_1$ & $\mathcal{C}_2$ & $\mathcal{C}_3$ & geometric setup \\
		1&NO & NO & NO & $K$-typical pair $(\tilde{\Delta},\tilde{C})$\\
		2&NO & NO & YES &  $g(\tilde{\Delta})=0$ and $g(\tilde{C})=1$\\
		3&NO & YES & NO & $g(\tilde{\Delta})=1$ and $\tilde{C}=B_1\cup B_2$, $\tau(B_1)=B_2$ and $g(B_i)=1$\\
		4&NO & YES & YES & $g(\tilde{\Delta})=0$ and $\tilde{C}=B_1\cup B_2$, $\tau(B_1)=B_2$ and $g(B_i)=0$\\
		5&YES & NO & NO & $g(\tilde{\Delta})=1$ and $g(\tilde{C})=2$\\
		6&YES & NO & YES & $g(\tilde{\Delta})=0$ and $g(\tilde{C})=0$\\
		7&YES & YES & NO & $g(\tilde{\Delta})=0$ and $\tilde{C}=B_1\cup B_2$, $\tau(B_1)=B_2$ and $g(B_i)=0$\\
		8&YES & YES & YES & $\tilde{\Delta}=L_1\cup L_2$ and $\tilde{C}=L'_1\cup L'_2\cup L'_3\cup L'_4$, $g(L_i)=g(L_j')=0$
	\end{tabular}\caption{A description of a typical geometric setup under each set of conditions.}\label{tab:non_generic_cases}
\end{table}

In line $3$, the curve $\tilde{C}$ decomposes into two irreducible components $B_i$, which (if geometrically irreducible) are curves of genus $1$ themselves. They are permuted by the appropriate involution $\tau$ acting on $\tilde{C}$. A similar phenomenon happens for the case no. 4 and no. 7. In line $8$ we denote by $L_i$ a line and by $L_i'$ a genus $0$ curve. We explain below where does the geometric setup in each case of Table \ref{tab:non_generic_cases} come from.

\textbf{Case 1.}
$\mu_{2,3}\neq0$, $\mu\neq0$ and $(a_0,\ldots, b_3)\not\in\rho(\mathcal{S})$.
This is the $K$-generic situation dealt with in the main theorem of this paper.

\textbf{Cases 2, 3, 4 and 6.}
In each of these setups, it is enough to assume that $P$ is a polynomial of degree $3$ and has the form $x(x-a)(x-b)$ for some $a,b\in\overline{K}$. We form a function field $F$ in the remaining coefficients over $K$ and argue over that field that the conditions of the ''geometric setup'' hold true. The decomposition of $\tilde{C}$ into irreducible components can be written explicitly and the genera computations were performed using MAGMA. We point the interested reader to our ancillary files attached to this manuscript, cf. \cite{Magma_computations}.

\textbf{Case 5.}
$\mu_{2,3}=0$, $\mu\neq0$ and $(a_0,\ldots, b_3)\not\in\rho(\mathcal{S})$.
The setup of no. 5 follows from the proof of Proposition \ref{prop:typic_imp_generic}.

\textbf{Case 7.}
$\mu_{2,3}=0$, $\mu=0$ and $(a_0,\ldots, b_3)\not\in\rho(\mathcal{S})$.
Without loss of generality we assume that $P$ has degree $3$ and working over $\overline{K}$ we assume that $P(x)=x(x-a)(x-b)$. It follows from the conditions $\mathcal{C}_1$ and $\mathcal{C}_2$ that 
\[\tilde{\Delta}: a (b - x_1 - x_2) - b (x_1 + x_2) + x_1^2 + x_1 x_2 + x_2^2\]
and the curve is a non-singular conic because the condition $C_3$ is equivalent to $(a^2-ab+b^2)=0$, which does not hold. Hence $\tilde{\Delta}$ is geometrically irreducible of genus $0$. The curve $\tilde{C}$ decomposes into two curves $B_1$ and $B_2$ with the following equations
\[B_1:\tilde{\Delta}(x_1,x_2)=0,\quad (a - x_2) (b - x_2) x_2=y,\]
\[B_2:\tilde{\Delta}(x_1,x_2)=0,\quad (a - x_2) (b - x_2) x_2=-y,\]
which for $\tau=\tau_1$ satisfy $\tau(B_1)=B_2$ and both are geometrically irreducible of genus~$0$.

\begin{example}
	Consider the polynomials $$P(x)=\frac{1}{5} \left(18 x^3+15 x^2+10 x+5\right),$$ $$Q(x) = \frac{1}{3} \left(18 x^3+15 x^2+10 x+12\right).$$ The generalized second moment of the family $\mathcal{E}_{U}: P(x)U+Q(x)=y^2$ is
	\[\widetilde{M}_{2,p}(\mathcal{E}) = 2p^2+p(-\left(\frac{-3}{p}\right)-\left(\frac{-35}{p}\right)-1-\mu_1(p)-\mu_2(p)+\mu_3(p))
	\]
	where $$\mu_1(p) = \#\{x\in\mathbb{F}_{p}: 5+10x+15x^2 +18x^3=0\},$$
	$$\mu_2(p)=\#\{x\in\mathbb{F}_{p}:10300+ 11025x^2+22680x^4 + 11664x^6=0\}$$ 
	and
	\begin{equation*}
	\begin{split}\mu_3(p)=(\#\{x\in\mathbb{F}_{p}:18/5x^3 + 3x^2 + 2x + 1=0\})^2+\\
	((1+\left(\frac{21}{p}\right))^2(\#\{x\in\mathbb{F}_{p}:18/5x^3 + 3x^2 + 2x + 1=0\}))^2\\
	-2((1+\left(\frac{21}{p}\right))(\#\{x\in\mathbb{F}_{p}:18/5x^3 + 3x^2 + 2x + 1=0\}))^2
		\end{split}
    \end{equation*}
	for all primes except $p\in \{2, 3, 5, 7,97, 103\}$.
	
	Notice, that in this example the leading term is $2p^2$ which is not in agreement with the theorem of Michel. This is due to the fact that the family $\mathcal{E}_{U}$ has non-minimal singular fibers.
	
\end{example}

\textbf{Case 8.}
$\mu_{2,3}=0$, $\mu=0$ and $(a_0,\ldots, b_3)\in\rho(\mathcal{S})$. Suppose that $P$ is not proportional to $Q$. Assume that $K$ is such that $\charac K\neq 2,3$.
In this case $\tilde{\Delta}$ is $K$-isomorphic to the reducible conic $X^2+XY+Y^2=0$ and $\tilde{C}$ is $K$-isomorphic to the union of $4$ schemes $X^2+XY+Y^2=0, \quad y^2=(a_0 + a_3 X^3)^2$ for certain $a_0,a_3\in K$. Let $K=\mathbb{Q}$ and $a_0,a_3\in\mathbb{Z}$. Let $p$ be a prime such that $\mu_{0,3}\not\equiv 0(\textrm{mod }p)$, $a_3\not\equiv 0(\textrm{mod }p)$. Note that these conditions are true for every sufficiently large prime. Then
\[\tilde{M}_{2,p} = p^2\left(2+\left(\frac{-3}{p}\right)\right)-1-\left(\frac{-3}{p}\right).\]
The family $\mathcal{E}_{U}$ has $j$-invariant $0$ and no non-minimal fiber, hence the Bias conjecture holds with bias equal to $-1$.

For a particular choice of polynomials $P, Q$ the $j$-invariant of the corresponding family $\mathcal{E}_{U}$ can be constant in $K$. The next proposition proves that the possibilities are very limited.
\begin{proposition}
	Suppose that the $j$-invariant of the elliptic curve $E_{P,Q}:P(x)U+Q(x)=y^2$ is constant, i.e., $j\in K$ for the polynomials $P,Q\in K[x]$. Then $j\in\{0,12^3\}$ or otherwise $P(x)=\lambda Q(x)$ for some non-zero $\lambda$. In each case, either $P,Q$ have at least one common root or we are in one of the cases $4$--$8$ of Table \ref{tab:non_generic_cases}.
\end{proposition}
\begin{proof}
	Since a $j$-invariant does not depend on the choice of a model of the elliptic curve $E_{P,Q}$ we can assume without loss of generality that $P(x)=x(x-1)(x-a)$, $Q(x)=b_0+b_1 x+b_2 x^2+b_3x^3$. We match the $j$-invariant of such a family with a variable $c$ and solve the Gr\"{o}bner basis problem in the variables $b_0,b_1,b_2,b_3,a,c$ determined by the vanishing of the expression $j-c$. The geometrically reduced subscheme corresponding to this setup has either $c=0$ or $c=12^3$ or $P(x)=Q(x)$.  The final claim follows from a straightforward case analysis of the irreducible components of the parametrizing schemes.
\end{proof}

\begin{remark}
    Some other non-typical cases of the Bias conjecture were proved in \cite{Novak_rector_award}.
\end{remark}

\bibliographystyle{alpha}
\bibliography{bibliography}
\end{document}